\documentclass[10pt, reqno]{amsart}
\usepackage{amsrefs, amsmath, amssymb, amsthm, graphicx,marvosym}
\usepackage{cancel}
\usepackage{cleveref}
\usepackage{commath}
\usepackage[normalem]{ulem}
\usepackage{mathtools}

\def\({\left(\begin{array}{cccccc}}
\def\){\end{array}\right)}

\def\({\left(\begin{array}{cccccc}}
\def\){\end{array}\right)}

\def\bes{\begin{eqnarray}}
\def\ees{\end{eqnarray}}

\def\bel{\begin{equation}\label}

\newcommand{\beq}{\begin{equation}}
\newcommand{\eeq}{\end{equation}}
\newcommand{\bea}{\begin{eqnarray}}
\newcommand{\eea}{\end{eqnarray}}
\newcommand{\beann}{\begin{eqnarray*}}
\newcommand{\eeann}{\end{eqnarray*}}

\newcommand{\RR}{\mathbb{R}}
\newcommand{\R}{\ensuremath{\mathbf{R}}}
\newcommand{\NN}{\mathbb{N}}

\newcommand{\Sbv}{\mathcal{S}_{BV,\varepsilon}}
\newcommand{\eps}{\varepsilon}
\newcommand{\Sweak}{\mathcal{S}_{weak}}
\newcommand{\SweakT}{\mathcal{S}^T_{weak}}
\newcommand{\Sl}{\mathcal{S}^T_{reg}}
\newcommand{\Nuo}{\mathcal{V}_0}
\newcommand{\Nu}{\mathcal{V}}
\newcommand{\ds}{\displaystyle}

\newcommand{\bp}{\begin{proof}}
\newcommand{\ep}{\end{proof}}

\newtheorem{theorem}{Theorem}[section]
\newtheorem{proposition}[theorem]{Proposition}

\newtheorem{lemma}[theorem]{Lemma}
\newtheorem{definition}[theorem]{Definition}
\newtheorem{assumption}[theorem]{Assumption}

\newtheorem{condition}[theorem]{Condition}

\newtheorem{remark}[theorem]{Remark}

\numberwithin{equation}{section}

\begin{document}

\title[weak-BV stability]{Uniqueness and weak-BV stability  for $2\times 2$ conservation laws}

\author[Chen]{Geng Chen}
\address[Geng Chen]{\newline Department of Mathematics, \newline University of Kansas, Lawrence, KS, 66049, USA}
\email{gengchen@ku.edu}

\author[Krupa]{Sam G.  Krupa}
\address[Sam G. Krupa]{\newline Max Planck Institute for Mathematics in the Sciences, \newline Inselstr. 22, 04103 Leipzig, Germany}
\email{Sam.Krupa@mis.mpg.de}

\author[Vasseur]{Alexis F. Vasseur}
\address[Alexis F. Vasseur]{\newline Department of Mathematics, \newline The University of Texas at Austin, Austin, TX 78712, USA}
\email{vasseur@math.utexas.edu}


\date{\today}

\thanks{\textbf{Acknowledgment.} 
G. Chen is partially supported by NSF with grants DMS-1715012 and DMS-2008504.  S. G. Krupa was partially supported by NSF-DMS Grant 1840314. A. F. Vasseur is partially supported by the NSF grant: DMS 1614918.  
}

\begin{abstract}
Let a  1-d system of hyperbolic conservation laws, with two unknowns, be endowed with a convex entropy. We consider the family of small $BV$ functions which are global solutions of this equation. 
For any small $BV$ initial data, such global solutions are known to exist. Moreover, they are known to be unique among $BV$ solutions verifying either the so-called Tame Oscillation Condition, or the  Bounded Variation Condition on space-like curves. 
 In this paper, we show that these solutions are stable in a larger class of weak (and possibly not even $BV$) solutions of the system. This result extends the classical weak-strong uniqueness results which allow comparison to a smooth solution. Indeed our result extends these results to a weak-$BV$ uniqueness result, where only one of the solutions is supposed to be small $BV$, and the other solution can come from a large class.
 As a consequence of our result, the Tame Oscillation Condition, and the Bounded Variation Condition on space-like curves are not necessary for the  uniqueness of solutions in the $BV$ theory, in the case of systems with 2 unknowns. The method is $L^2$ based. It builds up from  the theory of a-contraction with shifts, where  suitable weight functions $a$ are generated via the front tracking method.
\end{abstract}

\maketitle
\tableofcontents



\section{Introduction}
We consider 1-d system of hyperbolic conservation laws with two unknowns
\beq\label{cl}
u_t+(f(u))_x=0  \qquad t>0, \ x\in \RR,
\eeq
where $(t,x)\in {\mathbb R}^+\times \mathbb R$ are time and space, and $u=(u_1,u_2)\in \Nuo\subseteq \RR^2$ is the unknown. The set of states $\Nuo$ is supposed to be bounded, and we denote $\Nu$ its interior. Then $f=(f_1,f_2)\in [C(\Nuo)]^2\cap [C^4(\Nu)]^2$ is the flux function, and is assumed to be continuous on $\Nuo$ and $C^4$ on $\Nu$.   For any $g\in C^1(\Nu)$, let us  denote  the vector valued function $g'=Dg$.  Then,
we denote the eigenvalues and associated right eigenvectors of $f'$ on $\Nu$ as $\lambda_1$, $r_1$ and $\lambda_2$, $r_2$,  corresponding to the 1 and 2 characteristic families respectively. 
Through the paper, we will make the following general assumptions on the system.
\begin{assumption} Assumptions on the system\\    \label{assum}
\begin{itemize}
\item[(a)] For any $u\in \Nu$:  $\lambda_1(u)<\lambda_2(u)$.
\vskip0.1cm
\item[(b)]  For any $u\in\Nu$, and $i=1,2$:  $\lambda'_i(u)\cdot r_i(u)\neq0$.
\vskip0.1cm
\item[(c)] There exists a strictly convex  function $\eta\in C(\Nuo)\cap C^3(\Nu)$ and a function $q\in C(\Nuo)\cap C^3(\Nu)$ such that
\begin{equation}\label{eq:entropy}
q'=\eta'f', \qquad \mathrm{on} \ \ \Nu.
\end{equation}
\vskip0.1cm
\item[(d)] For any $ b\in \Nu$, and any left eigenvector  $\ell$ of $ f'(b)$:  the function $u\to \ell\cdot f(u)$ is either convex or concave on $\Nu$.
\item[(e)] There exists $L>0$ such that for any $u\in \Nu$ and $i=1,2$: $|\lambda_i(u)|\leq L$.
\item[(f)] For $u_L\in \Nu$, we denote $s\to  S^1_{u_L}(s)$ the $1$-shock  curve through $u_L$ defined for $s>0$. We choose the parametrization such that $s=|u_L-S^1_{u_L}(s)|$. Therefore, $(u_L, S^1_{u_L}(s), \sigma^1_{u_L}(s))$ is the $1$-shock with left hand  state $u_L$ and strength $s$. Similarly, we define $s\to S^2_{u_R}$ to be the $2$-shock curve such that $(S^2_{u_R}, u_R, \sigma^2_{u_R}(s))$ is the $2$-shock with right hand state $u_R$ and strength $s$. We assume that these curves are defined globally in $\Nu$ for every $u_L\in \Nu$ and $u_R\in \Nu$.
\item[(g)] (for 1-shocks) If $(u_L,u_R)$ is an entropic Rankine-Hugoniot discontinuity with shock speed $\sigma$, then $\sigma>\lambda_1(u_R).$
\item[(h)] (for 1-shocks) If $(u_L,u_R)$ (with $u_L\in B_{\epsilon}(d)$) is an entropic Rankine-Hugoniot discontinuity with shock speed $\sigma$ verifying,
\begin{align*}
\sigma\leq \lambda_1(u_L),
\end{align*}
then $u_R$ is in the image of $S^1_{u_L}$. That is, there exists $s_{u_R}\in[0,s_{u_L})$ such that $S^1_{u_L}(s_{u_R})=u_R$ (and hence $\sigma=\sigma^1_{u_L}(s_{u_R})$).
\item[(i)] (for 2-shocks) If $(u_L,u_R)$ is an entropic Rankine-Hugoniot discontinuity with shock speed $\sigma$, then $\sigma<\lambda_2(u_L).$
\item[(j)] (for 2-shocks) If $(u_L,u_R)$ (with $u_R\in B_{\epsilon}(d)$) is an entropic Rankine-Hugoniot discontinuity with shock speed $\sigma$ verifying,
\begin{align*}
\sigma\geq \lambda_n(u_R),
\end{align*}
then $u_L$ is in the image of $S^2_{u_R}$. That is, there exists $s_{u_L}\in[0,s_{u_R})$ such that $S^2_{u_R}(s_{u_L})=u_L$ (and hence $\sigma=\sigma^2_{u_R}(s_{u_L})$).
\item[(k)]  For $u_L\in \Nu$, and  for all $s>0$,  $\ds{\frac{d}{ds}\eta(u_L | S^1_{u_L}(s))}>0$ (the shock ``strengthens" with $s$).
Similarly, for $u_R\in \Nu$, and for all $s>0$, $\ds{\frac{d}{ds}\eta(u_R | S^2_{u_R}(s))}>0$. Moreover, for each $u_L,u_R\in \Nu$ and $s > 0$, $\frac{d}{ds}\sigma^1_{u_L}(s) < 0$ and $\frac{d}{ds}\sigma^2_{u_R}(s) > 0$.
\end{itemize}
\end{assumption} 
These assumptions are fairly general. The first one corresponds to the strict hyperbolicity of the system in $\Nu$. The second one means that both characteristics families of the system  are genuinely nonlinear  in $\Nu$ in the sense of Lax \cite{MR0093653}. The third assumption is related to the second law of thermodynamics. The function 
$\eta$ is called an entropy of the system, and $q$ is called the entropy flux associated with $\eta$. The next two assumptions are less classical. Assumption (d) ensures a contraction property on rarefaction waves (see Section \ref{sec:onewave}). Assumption (e) provides a global bound on the speeds of propagation. Assumptions (f) to (k) are now standard for the $a$-contraction theory. It was showed in \cite{Leger2011} that they are verified for a large family of systems, including the Full Euler system and the isentropic Euler system.
\vskip0.1cm
A typical  example of systems verifying the Assumptions \ref{assum} is the  system of isentropic Euler equations  for  $\gamma>1$:
\begin{equation}\label{eq:Euler}
\begin{array}{l}
\rho_t+(\rho v)_x=0, \qquad t>0, x\in \RR,\\
(\rho v)_t+(\rho v^2+\rho^\gamma)_x=0, \qquad t>0, x\in\RR,
\end{array}
\end{equation}
endowed with the physical entropy $\eta(u)=\rho v^2/2+\rho^{\gamma}/(\gamma-1)$, where $u=(\rho,\rho v)$. For any fixed constant $C>0$, we can define the space of states as the  invariant region:
\begin{equation}\label{InvReg}
\Nuo=\{ u=(\rho, \rho v)\in \RR^+\times\RR: \ \ -C< w_1(u)=v-c_1\rho^\frac{\gamma-1}{2} \leq  w_2(u)=v+c_1\rho^\frac{\gamma-1}{2}<C \}.
\end{equation}
Note that $\Nuo=\Nu \cup \{(0,0)\}$ where $(0,0)$ is the vacuum state. It justifies the  precise distinction  of $\Nu$ and $\Nuo$  (see  \cite{VASSEUR2008323, Leger2011}).  
The fact that \eqref{eq:Euler} verifies the three first assumptions of Assumption \ref{assum} is well-known (see Serre \cite{Serre_Book} for instance). We prove in Section \ref{sec:onewave} that it also verifies assumptions (d) and (e).

\vskip0.3cm
We will consider only entropic solutions of (\ref{cl}), that is, solutions which verify additionally 
\begin{equation}\label{ineq:entropy}
(\eta(u))_t+(q(u))_x\leq 0, \qquad t>0, \ x\in \RR.
\end{equation}

More precisely, we ask that for all $\phi\in C^\infty_0([0,\infty)\times\mathbb{R})$ verifying $\phi\geq 0$,
\begin{align}\label{entropy_ineq_integral_form}
\int\limits_0^\infty\int\limits_{-\infty}^\infty \Big[ \phi_t(t,x) \eta(u(t,x))+\phi_x(t,x) q(u(t,x)) \Big]\,dxdt+\int\limits_{-\infty}^\infty \phi(0,x) \eta(u^0(x))\,dx \geq 0,
\end{align}
where $u^0:\mathbb{R}\to\mathbb{R}$ is the prescribed initial data for the solution $u$. 
\vskip0.1cm
We also restrict our study to the solutions verifying the so-called Strong Trace Property.
\begin{definition}[Strong Trace Property]\label{defi_trace}
Let $u\in L^\infty(\RR^+\times\RR)$. We say that $u$ verifies the strong trace property if for any Lipschitzian curve $t\to X(t)$, there exists two bounded functions $u_-,u_+\in L^\infty(\RR^+)$ such that for any $T>0$
$$
\lim_{n\to\infty}\int_0^T\sup_{y\in(0,1/n)}|u(t,X(t)+y)-u_+(t)|\,dt=\lim_{n\to\infty}\int_0^T\sup_{y\in(-1/n,0)}|u(t,X(t)+y)-u_-(t)|\,dt=0.
$$
\end{definition}
For convenience, we will use  later the notation $u_+(t)=u(t, X(t)+)$, and $u_-(t)=u(t, X(t)-)$.
We can then define  the wildest space of solutions that we consider in the paper:
\begin{equation}\label{eq:weaksol}
\Sweak=\{u\in L^\infty(\RR^+\times \RR:\Nuo) \ \ \mathrm{weak \ solution \ to } \ \eqref{cl}\eqref{ineq:entropy}, \ \ \text{verifying \Cref{defi_trace}}  \}.
\end{equation}
Note that this space has no smallness condition. 
\vskip0.3cm
The aim of this paper is to show the stability  of a smaller class of solutions -namely solutions with small BV norms- when perturbations are taken in the wider space  $\Sweak$. More precisely, for any domain $\mathcal{O}$ such that $\overline{\mathcal{O}}\subset\Nu$, consider the following class of solutions:
\begin{equation}\label{eq:bvsol}
\Sbv=\{u\in L^\infty(\RR^+, BV(\RR:\mathcal{O})) \ \ \mathrm{solution \ to } \ \eqref{cl}\eqref{ineq:entropy}, \ \ \mathrm{with} \ \|u(t)\|_{BV(\RR)}\leq \varepsilon \ \mathrm{for} \   t\geq0  \}.
\end{equation}
Our main result is the following theorem. 
\begin{theorem}\label{theo}
Consider a system \eqref{cl}  verifying all the Assumptions \ref{assum}.
Then, for any  open set $\mathcal{O}$    such that  $\overline{\mathcal{O}}\subset\Nu$, there exists $\eps>0$ such that the following is true. 
\vskip0.1cm
Let  $u\in \Sbv$ be a BV solution with 
 initial value $u^0$.  
Assume that $u_n\in \Sweak$  is a sequence of wild solutions, 
 uniformly bounded in $L^\infty(\RR^+\times\RR)$, with initial values $u_n^0\in L^\infty(\RR)$.  If  $u^0_n$ converges to $u^0$ in $L^2(\RR)$, then 
for every $T>0, R>0$, $u_n$ converges to $u$ in $L^\infty(0,T; L^2(-R,R))$. Especially, $u$ is unique in the class $\Sweak$. 
\end{theorem}

In a celebrated paper \cite{gl},
Glimm showed that for any $\mathcal{O}$ compact subset of $\Nu$, any $a\in \mathcal{O}$, and  $\eps$ small enough, there exists $\eps_{in}>0$ such that if $\|u^0-a\|_{L^\infty(\RR)}\leq \eps_{in}$ and $\|u^0\|_{BV(\RR)}\leq \eps_{in}$, 
 then there exists a solution $u\in \Sbv$ of \eqref{cl} with $u(0,\cdot)=u^0$. After Glimm's method, currently referred to as the Glimm scheme or random choice method, there are two other frameworks which can be used to prove the small BV existence for general hyperbolic conservation laws: the front tracking scheme (see \cite{Bbook, dafermos_big_book}) and the vanishing viscosity method  \cite{MR2150387}.

\vskip0.1cm
Uniqueness of these solutions   was established by Bressan and Goatin \cite{MR1701818} under the Tame Oscillation Condition. It improved an earlier theorem by Bressan and LeFloch \cite{MR1489317}. Uniqueness was also known to  prevail when the Tame Oscillation Condition is replaced by the assumption that the trace of solutions along space-like curves has bounded variation, see Bressan and Lewicka \cite{MR1757395}. We will refer to this condition as the Bounded Variation Condition (see \Cref{tame_def}). One can also find these theories in  \cite{Bbook} or \cite{dafermos_big_book}. 
\vskip0.3cm
These uniqueness theories, which work for general hyperbolic conservation laws with $n$ unknowns, all need some a priori assumption on the solutions, such as Tame Oscillation Condition or Bounded Variation Condition on space-like curves (see \cite{Bbook}).
\vskip0.3cm
Note that  any $BV$ function verifies the Strong Trace Property (\Cref{defi_trace}). Hence, any $BV$ solution to \eqref{cl}\eqref{ineq:entropy} belongs to  $\Sweak$. Therefore a consequence of \Cref{theo}, is that  
in the case of 2 unknowns, these a priori assumptions are not needed to obtain the uniqueness result. We formulate this result in the following theorem.
\begin{theorem}\label{cor}
Consider a system \eqref{cl}  verifying all the Assumptions \ref{assum}.
Then, for any  open set $\mathcal{O}$    such that  $\overline{\mathcal{O}}\subset\Nu$, there exists $\eps>0$ such that the following is true. 
\vskip0.1cm
Any solution in $\Sbv$ with initial value $u^0$ is unique among  the functions 
\begin{align*}
\{v \in L^\infty([0,T];BV(\RR)) \mathrm{ \ for \ all \ } T>0 \ \ \mathrm{and \ solution \ to } \ \eqref{cl}\eqref{ineq:entropy}\}
\end{align*}
with the same initial value.
\end{theorem}
As another celebrated work  for small BV solutions of general systems of hyperbolic conservation laws, the $L^1$ stability has been established in 1990s  \cite{MR1723032,MR1686652}, or see \cite{Bbook}. 
In the $L^1$ stability theory, the  perturbations $u_n$ have to stay in the space $\Sbv$. In contrast,  this is not required in our $L^2$ based theory.
For this reason, Theorem \ref{theo} can be seen as a  weak-BV  stability result, similar to the weak-Strong stability result of  Dafermos and DiPerna.
\vskip0.3cm
Indeed, since the work of Dafermos  and DiPerna \cite{MR546634, MR523630}, it is known that on any span of time $[0,T]$ where a solution of the system is Lipschitz in $x$, the solution is $L^2$ stable (for $L^2$ perturbations on the initial value) among the large class of solutions which are bounded weak entropic solutions to the same system. This implies the  well known weak-strong uniqueness principle: as long as a solution is Lipschitz, it is unique among any other bounded weak solution. To be more precise, let us denote the two classes of solutions:
\begin{eqnarray*}
&&\Sl= \{ u\in L^\infty([0,T]\times\RR:\Nu), \ \ \mathrm{solution \ to } \ \eqref{cl}, \ \ \mathrm{with} \ \ \|\partial_xu(t)\|_{L^\infty([0,T]\times\RR)}\leq C,  \ \mathrm{for} \  C>0\}\\
&&\SweakT=\{u\in L^\infty([0,T]\times \RR:\Nuo) \ \ \mathrm{weak \ solution \ to } \ \eqref{cl}\eqref{ineq:entropy} \}.
\end{eqnarray*}
Let $\mathcal{O}$ be a compact subset of $\Nu$, and  $u$ be a solution in $\Sl$ with values in $\mathcal{O}$, and  with initial value $u^0$.
The result of Dafermos and DiPerna implies  that if $(u_n)_{n\in\NN}$ is a sequence of solutions in $\SweakT$ such that their initial values $(u_n^0)_{n\in \NN}$ converge in $[L^2(\RR)]^2$ to $u^0$, then $(u_n)_{n\in\NN}$ converges in $L^{\infty}(0,T;L^2(\RR))$ to $u$. Especially, it implies the uniqueness of solutions in $\Sl$ among the bigger class of solutions $\SweakT$ (weak/Strong uniqueness).  
\vskip0.3cm
Theorem \ref{theo} extends this result, in the context of 1-d systems with two unknowns, going from the Lipschitz space $\Sl$ to the $BV$ space $\Sbv$. 
Note, however, that the wild solutions of $\Sweak$ need to have the extra strong trace property compared to solutions of $\SweakT$.
Still, they can be valued in $\Nuo$, including states for which $f$ is not differentiable, like the vacuum for the Euler system \eqref{eq:Euler}.

\vskip0.3cm
Our work is motivated by the construction of bounded solutions via compensated compactness available for $2\times 2$ systems (see Tartar \cite{MR584398}, DiPerna \cite{MR719807}, Chen \cite{MR922139, MR924671}, Lions-Perthame and etc \cite{MR1284790,MR1383202}).  Those systems have Riemann invariants (see Smoller's book \cite{sm} and Hoff's paper \cite{MR784005})
$w_1$ and $w_2$, such that 
\[
(w_1)_t+\lambda_1(w_1)_x=0,\qquad 
(w_2)_t+\lambda_2(w_2)_x=0.
\]
This provides naturally   invariant regions as \eqref{InvReg}.
\vskip0.3cm
Strong traces properties were first proved for multivariable conservation laws \cite{MR1869441}, see also (\cite{Kwon,Yu}). The technique was later used to get more structural information on the  solutions (see \cite{Otto, Silvestre}). For systems, the question whether bounded weak solutions in $\SweakT$ verify the Strong Trace Property is mostly open. 
\vskip0.3cm
The Euler system \eqref{eq:Euler}  with $\gamma=3$ is an interesting case. Indeed,  it was proved in \cite{MR1720782} that for any initial values in $\Nuo$, one can construct global solutions with values in $\Nuo$ verifying a similar strong property in time. It would be interesting to investigate whether this property can be  extended to the property  of \Cref{defi_trace} in this context.
\vskip0.3cm
In a parallel program, it has been shown that, considering inviscid limit of Navier-Stokes equation, instead of  weak solution to the inviscid conservation laws, one can avoid the need of the strong traces property. The case of the inviscid limit of the barotropic  Navier-Stokes equation in the Lagrangian variables is considered in  \cite{2017arXiv171207348K, 2019arXiv190201792K}. It is shown that single shocks are stable (and so unique) in the class of inviscid weak limit of energy bounded solutions to  Navier-Stokes equations. Neither Boundedness of the function, nor the strong traces property are needed in this context. 
This result is a first milestone in the program of the authors to show the convergence from Navier-Stokes to Euler for initial values small in BV, a major open problem in the field \cite{Bbook}.  Theorem \ref{theo} is a second major milestone in this direction. It provides several tools needed in the program to leap from the study of a single wave solution to general Cauchy data.

\vskip0.3cm
The paper is structured as follows. We begin in Section \ref{section:prelim}  with preliminaries linked to the Bounded Variation Condition along space-like curve needed for the  $L^1$ uniqueness theory.  The proof of our result  is based on the weighted entropy method with shifts, and the front tracking method. The main ideas of our proof  are presented in Section \ref{sec:idea}. This section proves Theorem \ref{theo} from Proposition \ref{main_prop}.   The rest of the paper is dedicated to the proof of Proposition  \ref{main_prop}.  Section \ref{sec:onewave} is dedicated to the $L^2$ study of single waves. The most important one concerns the study of a single shock. 
The exact needed version, Proposition \ref{shift_existence_prop}, is proved in a companion paper \cite{companion}. 
The modified front tracking algorithm is introduced in Section \ref{section:front}.
The construction of the weight functions are performed in Section \ref{sec:Geng}. Finally, Section \ref{sec:Sam} is dedicated to the proof of Proposition \ref{main_prop}.

\section{Preliminaries} \label{section:prelim}


\vskip0.3cm
This section gather tools from the $L^1$ theory that will be useful later. Every result and notion for this section comes from \cite{Bbook}. 
Our proof uses the $L^1$ uniqueness result of \cite{MR1757395}. Let us write precisely the statement here.  Following \cite{Bbook}, we introduce first the notion of  space-like curve.
\begin{definition}[Space-like curves]\label{def:space}
Let $\hat{\lambda}$ be a fixed constant. 
Then we define a \emph{space-like curve} to be a curve of the form $\{t=\gamma(x):x\in(a,b)\}$, with
\begin{align}
\abs{\gamma(x_2)-\gamma(x_1)}< \frac{x_2-x_1}{\hat{\lambda}} \hspace{.3in} \mbox{for all } a <x_1<x_2<b.
\end{align}
\end{definition}
 In this paper, the value of $\hat{\lambda}$ will be determined by Proposition \ref{shift_existence_prop}. 
 Still following \cite{Bbook}, we now introduce the extra condition needed for the classical $L^1$ uniqueness theorem.  
\begin{definition}[Bounded Variation Condition]\label{tame_def} 
We say that a function $u\in L^\infty(\RR^+\times\RR)$ verifies the Bounded Variation Condition if 
there exists $\delta>0$ such that, for every bounded space-like curve $\{t=\gamma(x): x\in[a',b']\}$ with
\begin{align}
\abs{\gamma(x_1)-\gamma(x_2)}\leq \delta\abs{x_1-x_2}\hspace{.25in}\mbox{for all } x_1,x_2\in[a',b'],
\end{align}
the function $x\mapsto u(\gamma(x),x):=u_\gamma(x)$ is well defined and has bounded variation.
\end{definition}
Note that taking  constant functions $\gamma$ shows that these functions $u$ are $BV$ in $x$.
Let us now state a uniqueness result of  \cite{MR1757395,Bbook}, rephrased in our context.
\begin{theorem}[From \cite{Bbook,MR1757395}]\label{theo_Bressan}
 For any $d\in \Nu$, there exists $\eps>0$ such that for any $u^0$ initial value with $\|u^0\|_{BV(\RR)}\leq \eps$ and  $\|u^0-d\|_{L^\infty(\RR)}\leq \eps$, we have the following uniqueness result. 
 \vskip0.1cm
 There exists only one solution $u$ of \eqref{cl} \eqref{ineq:entropy} with initial value $u^0$ and verifying the Bounded Variation Condition of Definition \ref{tame_def}.
 
\end{theorem}
Note that \Cref{cor} replaces the condition of Definition  \ref{tame_def}, by only $u\in L^\infty(\RR^+;BV(\RR))$, and Theorem \ref{theo} by $u\in \Sweak$.
\vskip0.3cm
We will need to prove that a certain limit of solutions to a modified front tracking algorithm  inherits the Bounded Variation Condition. Still following \cite{Bbook}, we introduce  the following domination principle.
\begin{definition}[Domination]\label{domination_def}
Given two space-like curves $\gamma\colon (a,b)\to\mathbb{R}$ and $\gamma'\colon(a',b')\to\mathbb{R}$, we say that $\gamma$ \emph{dominates} $\gamma'$ if $a\leq a'<b'\leq b$ and, moreover,
\begin{align}\label{determinacy_eq}
\gamma(x)\leq \gamma'(x)\leq \min\Big\{\gamma(a)+\frac{x-a}{\hat{\lambda}},\gamma(b)+\frac{b-x}{\hat{\lambda}}\Big\} \hspace{.25in} \mbox{for all } x\in(a',b').
\end{align}
\end{definition}
This property implies  that $\gamma'$ is entirely contained in a domain of determinacy for the curve $\gamma$.
We introduce now the following property.
\begin{condition}\label{H} Let $C>0$.
Let a function $\psi\in L^\infty(\RR^+;BV(\RR))$ be  piecewise constant. We say that it verifies the Condition \ref{H} with constant $C$, if it verifies the following.
\vskip0.3cm
 Let $\gamma$ and $\gamma'$ be any two space-like curves 
 with $\gamma$ dominating $\gamma'$ (\Cref{domination_def}). Then,
\begin{align}\label{truc}
\rm{Tot. Var.}\{\psi;\gamma'\} \leq C \rm{Tot. Var.}\{\psi;\gamma\}.
\end{align}
\end{condition}
We will use the following lemma. 

\begin{lemma}\label{lem:limit}
Let $\{\psi_n\}_{n\in \NN}$ be a a family of piecewise constant functions uniformly bounded in $L^\infty(\RR^+,BV(\RR))$.
Assume that there exists $C>0$ such that   for every $n\in \NN$,  $\psi_n$ verifies Condition \ref{H} for this constant $C$, and 
\begin{equation}\label{trucmuche}
\norm{\psi_n(t,\cdot)-\psi_n(s,\cdot)}_{L^1}\leq C\abs{t-s},\qquad 0<s<t<T.
\end{equation}
Then, there exists $\psi\in L^\infty(\RR^+\times\RR)$ verifying the Bounded Variation Condition \ref{tame_def} such that,  up to a subsequence, $\psi_n$ converges  to $\psi$ when $n\to\infty$ in $C^0(0,T;L^2(-R,R))$  for every $T>0$, $R>0$, and almost everywhere in $\RR^+\times\RR$.
\end{lemma}

This lemma is very similar to \cite[Lemma 7.3]{Bbook}, where the same result is stated for classical  piecewise constant approximate solutions constructed by the front tracking algorithm (without shifts). For the sake of completeness, we provide a proof of it  in the appendix. 


\section{Weighted relative entropy and shifts}\label{sec:idea}


The proof of our result is based on the relative entropy method first introduced by Dafermos \cite{MR546634} and DiPerna \cite{MR523630}. From the assumption of the existence of a convex entropy $\eta$, we define an associated pseudo-distance defined for any $a,b\in \Nuo\times \Nu$:
\begin{align}
\eta(a|b)=\eta(a)-\eta(b)-\nabla\eta(b)(a-b).
\end{align}
The quantity  $\eta(a|b)$ is called the relative entropy of $a$ with respect to $b$, and is equivalent to $|a-b|^2$.  
We also define the relative entropy-flux: For $a,b\in\mathbb{R}^2$,
\begin{align}\label{def_q}
q(a;b)=q(a)-q(b)-\nabla\eta(b)(f(a)-f(b)).
\end{align}
The strength of this notion is that if $u$ is a weak solution of \eqref{cl}, \eqref{ineq:entropy}, then $u$ verifies also the full family of entropy inequalities  for any $b\in \Nu$ constant:
\begin{equation}\label{ineq:relative}
(\eta(u|b))_t+(q(u;b))_x\leq 0.
\end{equation}
Similar to the Kruzkov theory for scalar conservation laws, \eqref{ineq:relative} provides a full family of entropies measuring the distance of the solution to any fixed values $b$ in $\Nu$. The main difference is that the distance is equivalent to the square of the $L^2$  norm rather than the $L^1$ norm. Same as for the Kruzkov theory, \eqref{ineq:relative} provides directly the stability of constant solutions (by integrating in $x$ the inequalitiy). Modulating the inequality with a smooth function $t,x\to b(t,x) $ provides the well-known weak-strong uniqueness result. 
Precisely, the relative entropy is an $L^2$ theory in the following sense:
\begin{lemma}\label{l2_rel_entropy_lemma}
 For any fixed compact set $V\subset \mathcal{V}$, there exists $c^*,c^{**}>0$ such that for all $(u,v)\in \Nuo\times V$,
\begin{align}
c^*\abs{u-v}^2\leq \eta(u|v)\leq c^{**}\abs{u-v}^2.
\end{align}
The constants $c^*,c^{**}$ depend on  bounds on the second derivative of $\eta$ in $V$, and on the continuity of $\eta$ on $\Nuo$. 
\end{lemma}
This elementary  lemma follows directly from Taylor's theorem  (see \cite{Leger2011,VASSEUR2008323}).

\vskip0.3cm
For the family of Euler systems, it is well known that the relative entropy provides a contraction property for rarefaction function $t,x\to b(t,x)$, even in multi-D \cite{MR3357629}. This is because it verifies Assumption \ref{assum} (d)   (see \Cref{sec:onewave}). 
\vskip0.3cm 
However, when modulating the inequality with discontinuous functions $b$ with shocks, the situation diverges significantly from the Kruzkov situation. This is due to the fact that the $L^2$ norm is not as well suited as the $L^1$ norm  for the study of stability of shocks. However, the method was used by DiPerna \cite{MR523630} to show the uniqueness of single shocks (see also Chen Frid \cite{MR1911734} for the Riemann problem of the Euler equation).  In \cite{VASSEUR2008323}, it was proposed to use the method to obtain stability of discontinuous solutions. The main idea was that the $L^2$ norm can capture very well the stability of the profile of the shock (up to a shift), even if the shift itself is more sensitive \cite{Leger2011}.  Leger in \cite{Leger2011_original} showed that in the scalar settings, the shock profiles (modulo shifts) have a contraction property in $L^2$, reminiscent to the $L^1$ contraction of the Kruzkov theory.  It was shown in \cite{serre_vasseur} that the contraction property is usually false for systems. However, it can be recovered by weighting the relative entropy \cite{MR3519973}. 
More precisely, consider a fixed shock $(u_l,u_r, s)$. It was shown that there exists $0<a_1<a_2$ such that, for any wild solutions $u\in \Sweak$, we can construct a Lipschitz shift function $h:\RR^+\to\RR$ with 
\begin{equation}\label{eq:ex}
\frac{d}{dt}\left\{a_1 \int_{-\infty}^{h(t)} \eta(u(t,x)|u_L)\,dx+ a_2\int_{h(t)}^\infty \eta(u(t,x)|u_R)\,dx \right\} \leq0.
\end{equation}
Note that this formula for $a_1=a_2$, and $h(t)=st$ would imply the contraction property of the shock for the relative entropy. But the result, to be valid, needs the weights $a_i$, and the shifts $h$, giving the name to the method: a-contraction with shifts.
\vskip0.3cm
Let us emphasize that the  $L^2$ based  a-contraction is not true without the notion of shifts. This is a major obstruction to consider solutions with several waves. 
Conservation laws have finite speeds of propagation. Therefore, usually, considering a finite amount of waves is equivalent to studying a single one, at least, as long as they do not interact. Because of the shifts, it is not obvious anymore in this theory. The general idea, is that one shift by singularity is needed. Those shifts depend crucially on the perturbation. It is therefore needed to prevent that this artificial shifts do not force a  1-shock  to stick and holds to a 2-shock, making the whole process to collapse. This problem was solved in \cite{a_contraction_riemann_problem}, allowing the treatment of the Riemann problem. The main idea  is that the shifts can be constructed based on perturbed characteristic curves associated to the wild solution.
\vskip0.3cm
This article is making the leap going from the stability of the Riemann problem, to the stability of small BV solutions. Because of the generation of infinitely many shifts, 
the estimate \eqref{eq:ex} is significantly weakened in this case. Our main proposition is the following.
\begin{proposition}\label{main_prop}
Consider a system \eqref{cl}  verifying all the Assumptions \ref{assum}. Let $d\in \Nu$. Then there exist $C,v,\eps>0$ such that the following is true. 
\vskip0.1cm
For any  $m>0$, $R,T>0$.   $u^0\in BV(\RR)$  such that $\|u^0\|_{BV(\RR)}\leq \eps$ and $ \|u^0-d\|_{L^\infty(\RR)}\leq \eps$, and any wild solution $u\in \Sweak$,   there exists  $\psi\colon\mathbb{R}^+\times\mathbb{R}\to\Nu$ such that for almost every $0<s<t<T$:
\begin{eqnarray*}
&& \|\psi(t,\cdot)\|_{BV(\RR)}\leq C \|u^0\|_{BV(\RR)},\\
&&\norm{\psi(t,\cdot)-\psi(s,\cdot)}_{L^1}\leq C\abs{t-s},\\
&&\|\psi(t,\cdot)-u(t,\cdot)\|_{L^2((-R+vt,R-vt))}\leq C \left(\|u^0-u(0,\cdot)\|_{L^2(-R,R)}+\frac{1}{m}\right),\\
&& \text{The function } \psi \text{ verifies the Condition \ref{H} with constant  } C.
\end{eqnarray*}
\end{proposition}
It would be natural to try to take for the  function $\psi$,  the unique BV solution with initial value $u^0$ of Theorem \ref{theo_Bressan}. However, functions $\psi$ which verify the proposition are not solutions to \eqref{cl}. 
Instead, the proposition shows that if the initial value $u(0,\cdot)$ is $L^2$ close to a set of small BV functions, then $u(t,\cdot)$ stays $L^2$ close,  for every time $t>0$,  to a slightly bigger  set of small BV functions.
\vskip0.3cm

Despite the finite speed propagation of the equation, there are two major difficulties to obtain this result: one  is the shifts, the other is the weights. Let us give an example of the difficulties the shifts introduce. Consider a piecewise-constant  solution $\bar{u}\in  \Sbv$ to \eqref{cl} \eqref{ineq:entropy}. Until the first time that there is an interaction between the shocks in $\bar{u}$, we can represent the function $\bar{u}$ as
$$
\bar{u}(t,x)=u_i ,\qquad \mathrm{for \  } x_i+s_i t< x < x_{i+1}+ s_{i+1}t,
$$
where $(u_{i-1}, u_i, s_i)$ are admissible shocks. For any weak solution $u\in\Sweak$, the general theory of weighted relative entropy with shifts ensure the existence of shifts $t \to h_i(t)$ and a piecewise weighted function
$$
a(t,x)=a_i, \qquad \mathrm{for} \ h_{i-1}(t)\leq x \leq h_i(t),
$$
such that, as long as the functions $h_i(t)$ do not cross:
\begin{equation}\label{ineq:contraction}
\frac{d}{dt}\int_\RR a(t,x)\eta(u(t,x)|\psi(t,x))\,dx \leq0,
\end{equation}
where 
$$
\psi(t,x)= u_i \qquad \mathrm{for \  } h_{i-1}(t)\leq x_{i+1}\leq h_i(t). 
$$

Let  $t^*>0$ be such that there are no collisions between any of the wavefronts in either $\bar{u}$ or $\psi$ for $t\in[0,t^*]$.  One might then hope to control $\norm{\bar{u}(t,\cdot)-u(t,\cdot)}_{L^2}$ by using both \eqref{ineq:contraction} and by controlling $\norm{\bar{u}(t,\cdot)-\psi(t,\cdot)}_{L^2}$. However, it is difficult to control $\norm{\bar{u}(t,\cdot)-\psi(t,\cdot)}_{L^2}$. For time $t\in[0,t^*]$, the function $\psi$ can be reconstructed from the function $\bar{u}$ via a change of variables. However, if for example after time $t^*$ a collision between two waves occurs in $\bar{u}$, but the corresponding waves do not collide in $\psi$, then after the local Riemann problem in $\bar{u}$ is solved and the clock restarted, the functions $\bar{u}$ and $\psi$ \emph{cannot} be related through a change of variables. Furthermore, as in the scalar case, the best control we have on the shifts gets worse and worse as the strength of the shock (being controlled by the shift) decreases (see \cite[Theorem 1.1]{move_entire_solution_system} and \cite[Theorem 1.2]{scalar_move_entire_solution}). This is problematic, because we want the initial data of the function $\psi$ to approach the initial data for the wild solution $u$, and in general the sizes of the shocks in $\psi$ will get arbitrarily small. 

The idea then is to give up on trying to control $\bar{u}-\psi$ or the shifts.  Instead, we construct an artificial function $\psi$ which stays $L^2$ close to $u$ while sharing the structural property of $\bar{u}$ (the smallness in $BV$). If we now consider a sequence of such solutions $u_n$ such that the initial values converge to the initial value of $\bar{u}$, we can transfer, at the limit, the structural property from the $\psi_n$ functions to the limit $u$. This implies that $u$ belongs to $\Sbv$ and still verifies the Bounded Variation Condition,  and so is equal to $\bar{u}$ by the uniqueness theorem (\Cref{theo_Bressan}).
\vskip0.3cm
Note that this strategy was first introduced in \cite{2017arXiv170905610K} in the scalar case with convex fluxes. The paper \cite{2017arXiv170905610K} gives a new proof for the uniqueness of solutions verifying a single entropy condition. Previous proofs of this result were obtained by Panov \cite{Panov1994_translated} and De Lellis, Otto, and Westdickenberg \cite{delellis_uniquneness}. Their proofs were based on the link between conservation laws and Hamilton-Jacobi equations, and it seems difficult to extend them to the system case where no such relation exists.  
\vskip0.3cm
The strategy is now  to construct the  function $\psi$ of Proposition \ref{main_prop} which stays $L^2$ close to $u$, while enjoying the small $BV$ property of $\bar{u}$. We construct it via the front tracking method, from the initial value of $\bar{u}^0$, but with the propagation of discontinuities following the shifts $\{h_i\}$ (which depend on the weak solution $u$). A key point is that the $BV$ estimates obtained from the front tracking method do not depend on the propagation of these fronts. We can then control the $BV$ norm of $\psi$. Note that $\psi$ is not a solution to the equation (\ref{cl}) since the Rankin-Hugoniot conditions are not verified anymore. 
It cannot be easily compared to $\bar{u}$ either since the waves can interact in a different order. We remark that although we limit ourselves in this paper to $2\times2$ systems with genuinely nonlinear wave families, the front tracking algorithm we use works for general $n\times n$ systems with either genuinely nonlinear or linearly degenerate wave families. In particular, in this paper, we have incorporated non-physical waves (also known as pseudoshocks) into our modified front tracking algorithm. The use of non-physical waves is not required for the $2\times 2$ case, but it is necessary for applying the front tracking method to general $n\times n$ systems.
\vskip0.3cm
The last difficulty is due to the weight function $a$. In order to obtain the contraction property \eqref{ineq:contraction}, we have constraints on the variations of the weights $a_i-a_{i-1}$ which depend both on the size and  the family of the shock $(u_{i-1}, u_i,s_i)$. This means that the weight function has to be reconstructed at each collision time between two waves. The variations of the weight function has to be controlled precisely to ensure that it stays bounded away from 0. The construction of the weight is closely related to the front tracking method, and the control of its $BV$ norm mirrors the $BV$ control on the function $\psi$ itself. 
\vskip0.3cm
For this procedure, a  key  refinement of the a-contraction for a single shock is  provided by the companion paper \cite{companion}. It shows that the size of variation of the weight $|a_i-a_{i-1}|$ can be chosen proportionally to the strength of the associated shock wave $|\sigma_i|\approx |u_{i-1}-u_i|$. This property was first showed in the class of inviscid limit of Navier-Stokes \cite{2017arXiv171207348K, 2019arXiv190201792K}. Surprisingly, the proof based directly on the inviscid model is very different, and quite delicate. 
\vskip0.3cm
We finish this section showing how Proposition \ref{main_prop} implies Theorem \ref{theo}.

\vspace{0.3cm}
\noindent\emph{Proof of Main Theorem (\Cref{theo})}.
For each $d\in \overline{\mathcal{O}}$, consider  $\eps_d>0$ such that both  Proposition~\ref{main_prop} and Theorem  \ref{theo_Bressan} are valid. The union (over $d$) of the balls $B_{\eps_d/2}(d)$ cover the compact  $\overline{\mathcal{O}}$, so there exists a finite subcover. Denote $\eps>0$ the smallest of the $\eps_{d_i}/2$ for this finite subcover. 

\vskip0.1cm
By passing to a subsequence if necessary, we assume that $\norm{u_m^0-u^0}_{L^2} \leq \frac{1}{m}$.  From \Cref{main_prop} we have a sequence of functions  $\psi_m$ (for all $m\in\mathbb{N}$),  
  uniformly bounded in $L^\infty(\RR^+,BV(\RR))$. 
Moreover   $\psi_m$ verify Condition \ref{H} and \eqref{trucmuche} uniformly, and they verify for all time $t>0$:
\begin{equation}\label{gg}
\norm{\psi_m(t,\cdot)-u_m(t,\cdot)}_{L^2(\RR)}\leq \frac{2}{m}.
\end{equation}

From Lemma \ref{lem:limit}, there exists $\psi\in L^\infty(\RR^+\times\RR)$ verifying the Bounded Variation Condition (\Cref{tame_def}) such that 
for every $T>0, R>0$, $\psi_m$ converges in $C^0(0,T:L^2(-R,R))$ to $\psi$. Together with \eqref{gg},
$u_m$ converges in $L^\infty(0,T:L^2(-R,R))$ to $\psi$. Since the convergence is strong and $u_m$ verifies   \eqref{cl} \eqref{ineq:entropy}, the limit $\psi$ is also solution to \eqref{cl} \eqref{ineq:entropy}, with initial value $u^0$. From Theorem \ref{theo_Bressan}, it is the unique solution verifying \Cref{tame_def}.
\vskip0.3cm
Applying the result to the constant sequence $\overline{u}_n= u$, the fixed BV function with initial value $u^0$ from the hypotheses of the theorem, shows that $u$ is also this unique solution. Therefore $\psi=u$. This ends the proof of \Cref{theo}.
The rest of the paper is dedicated to the proof of Proposition  \ref{main_prop}.

\section{Relative entropy for the Riemann problem}\label{sec:onewave}

We  first state the refined a-contraction property of shocks for the weighted relative entropy with shifts. This result is proved in \cite{companion}. Note that the constant $L$ is defined in Assumption \ref{assum} (e).

\begin{proposition}\label{shift_existence_prop}
Consider a system \eqref{cl}  verifying all the Assumptions \ref{assum}.
Let $d\in \Nu$. Then there exist constants $\alpha_1,\alpha_2, \hat{\lambda}$ and  $C, \eps>0$, with   $\alpha_1<\alpha_2$ and  $\hat{\lambda}\geq 2L$,  such that the following is true.\vskip0.1cm

Consider any shock  $(u_L,u_R)$   with $|u_L-d|+|u_R-d|\leq \eps$, any $u\in \Sweak$, any $\bar{t}\in[0,\infty)$, and any $x_0\in \RR$.
 Let  $\sigma$ be the strength of the shock $\sigma=|u_L-u_R|$.
Then for any $a_1>0,a_2>0$ verifying 
\begin{eqnarray*}
1- 2C\sigma\leq \frac{a_2}{a_1}\leq 1-\frac{C\sigma}{2} &&\text{ if } (u_L,u_R) \text{ is a 1-shock}\\
1+\frac{C\sigma}{2}\leq \frac{a_2}{a_1}\leq 1+2C\sigma && \text{ if } (u_L,u_R) \text{ is a 2-shock},
\end{eqnarray*}
there exists a Lipschitz shift function  $h:[\bar{t},\infty)\to\mathbb{R}$, with $h(\bar{t})=x_0$,  such that the following dissipation functional verifies
\begin{align} \nonumber
&a_1\left[q(u(t,h(t)+);u_R)-\dot{h}(t) \ \eta(u(t,h(t)+)|u_R)\right]\\ \label{diss:shock}
&\qquad\qquad  -a_2\left[q(u(t,h(t)-);u_L)-\dot{h}(t) \ \eta(u(t,h(t)-)|u_L)\right] \leq 0
\end{align}
for almost all  $t\in[\bar{t},\infty)$.

\vskip0.2cm

Moreover, if $(u_L,u_R)$ is a 1-shock, then for almost all  $t\in[\bar{t},\infty)$:  
$$
-\frac{\hat{\lambda}}{2}\leq \dot{h}(t) \leq \alpha_1<\inf_{v\in B_{2\eps}(d)}\lambda_2(v).
$$
Similarly, if $(u_L,u_R)$ is a 2-shock,  then for almost all  $t\in[\bar{t},\infty)$: 
$$
\sup_{v\in B_{2\eps}(d)}\lambda_1(v)< \alpha_2\leq \dot{h}(t) \leq\frac{\hat{\lambda}}{2}.
$$



\end{proposition} 


Integrating \eqref{ineq:relative}  with $b= u_L$ for  $x\in (-\infty, h(t))$, integrating \eqref{ineq:relative} with $b= u_R$ for $x\in (h(t),\infty)$, summing the results, and using \eqref{diss:shock} together with the strong traces property \Cref{defi_trace} provides the contraction property \eqref{ineq:contraction} in the case of a single shock as long as,
$$
a_2/a_1 \text{ is between } 1+\frac{C}{2}(-1)^i\sigma \text{ and } 1+2C(-1)^i\sigma,
$$
when $(u_L,u_R)$ is a $i$-shock.
It shows that the variation of the $a$ function has to be negative for a 1-shock, positive  for a 2-shock, and   can be chosen with strength of the same order as the size of the shock. The estimates on $\dot{h}$ show that we keep a finite speed of propagation, and 
that a shift of a 1-shock cannot overtake the shift of a 2-shock if it started on its left. This is  important because when we introduce shifts into the solution to a Riemann problem with two shocks,  both shock speeds move with  artificial velocities. We need to ensure that the positions of the shocks do not touch at some time after the initial time to preserve the property of  classical solutions to the Riemann problem, where shocks born from a solution to a Riemann problem will never touch.
\vskip0.3cm
We need a similar control for approximations of rarefactions via the  front tracking method. We begin to show that, under the Assumption \ref{assum} (b)(d), the real rarefaction has a contraction property without  the need of shift.

\begin{lemma}\label{lem:rar}
Consider a system \eqref{cl}  verifying all the Assumptions \ref{assum}.
Let $\bar{u}(y)$ $v_L\leq y\leq v_R$ be a rarefaction wave for \eqref{cl}. Then for any $u\in \Sweak$ and every $t>0$ we have
\begin{eqnarray*}
\frac{d}{dt} \int_{v_Lt}^{v_Rt}\eta(u(t,x)|\bar{u}(x/t))\,dx&\leq& q(u(t,v_Lt+);\bar{u}(v_L))-q(u(t,v_Rt-)|\bar{u}(v_R))\\
&&-v_L\eta(u(t,v_Lt+)|\bar{u}(v_L))+v_R\eta(u(t,v_Rt-)|\bar{u}(v_R)).
\end{eqnarray*}
\end{lemma}
\begin{remark}
We generalize the result and proof known for Euler. For one possible reference for this, see \cite{VASSEUR2008323}.
\end{remark}
\begin{proof}
Following \cite{MR546634} or \cite{VASSEUR2008323} we have that for any $u\in \Sweak$ and any  $v\in\Sl$:
$$
\partial_t \eta(u|v)+\partial_x q(u;v) +\partial_x\{\eta'(v)\}f(u|v)\leq 0.
$$
The relative flux is defined analogously to the relative entropy. For $a,b\in \Nuo\times \Nu$, we define it as:
\begin{align*}
f(a|b)=f(a)-f(b)-\nabla f(b)(a-b).
\end{align*}
Since $\bar{u}$ is a rarefaction, for all $y$, $\bar{u}'(y)$ is a right eigenvector of $f'(\bar{u}(y))$, and so from Assumption \ref{assum} (b): $\bar{u}'(y)\cdot \lambda'(\bar{u}(y))>0$.
For $v(t,x)=\bar{u}(x/t)$,  $\partial_xv=\frac{1}{t}\bar{u}'(x/t)=r$ is also a right eigenvector of $f'(\bar{u}(y))=f'(v)$ with the same property  from Assumption \ref{assum} (b):
\begin{equation}\label{genuine}
r\cdot  \lambda'(v)>0.
\end{equation}
Denote $\ell=\partial_x\{\eta'(v)\}=\eta''(v)\partial_xv=\eta''(v) r$. 
For $v$ fixed, since $\eta''$ is positive
\begin{equation}\label{positive}
\ell\cdot r=r^T\eta''(v)r>0.
\end{equation}
Since $\eta''(v)f'(v)$ is symmetric
$$
\ell^T f'(v)=r^T\eta''(v)f'(v)=\eta''(v)f'(v) r=\lambda(v) \eta''(v)r=\lambda(v)\ell.
$$
 Hence $\ell$ is a left eigenvector of $f'(v)$.
From Assumption \ref{assum} (d), we have that $\ell\cdot f$ is either convex or convave.
Let $r(u)$ be a right eigenvector of $f'(u)$ for the same family. Using that $\ell$ is a left eigenvector of $f'(v)$, we get
$$
\ell^T(f'(u)-f'(v))r(u)= \ell\cdot r(u) (\lambda(u)-\lambda(v)).
$$
passing to the limit $u$ goes to $v$, and taking the value along $r=r(v)$ we find:
$$
\ell f''(v)(r,r)=(\lambda'(v)\cdot r)(\ell\cdot r)>0
$$
thanks to \eqref{genuine} and \eqref{positive}. Hence, $\ell\cdot f$ is convex, and 
$$
u\to \partial_x\{\eta'(v)\}f(u|v)
$$
is non-negative for all $u\in \Nu$. Finally
$$
\partial_t \eta(u|v)+\partial_x q(u;v) \leq0.
$$
Integrating in $x$ between $v_Lt$ and $v_R t$, and using the Strong Trace Property (\Cref{defi_trace})  gives the result.
\end{proof}
We can now give the control needed for the error due to the approximation of the rarefaction via the front tracking method.

\begin{proposition}\label{dissipation_discrete_rarefaction}
There exists a constant $C>0$ such that the following is true.
For any  $\bar{u}(y)$ $v_L\leq y\leq v_R$  rarefaction wave for \eqref{cl}, denote 
$$
\delta=|v_L-v_R|+\sup_{y\in [v_L,v_R]} |u_L-\bar{u}(y)|, \qquad \bar{u}(v_L)=u_L, \ \bar{u}(v_R)=u_R.
$$
 Then for any $u\in \Sweak$, any $v_L\leq v\leq v_R$,  and any $t>0$ we have:
 $$
 \int_{0}^{t} \left\{
q(u(t,tv+);u_R)-q(u(t,tv-);u_L)-v\left( \eta(u(t,tv+)|u_R) -\eta(u(t,tv-)|u_L) \right)
 \right\}\,dt
 \leq C\delta |u_L-u_R|
 t.
 $$
\end{proposition}
\begin{proof}
Consider the quantity 
\begin{eqnarray}\nonumber
&&D=\frac{d}{dt}\left\{\int_{v_Lt}^{vt} \eta(u(t,x)|u_L)\,dx+\int_{vt}^{v_Rt} \eta(u(t,x)|u_R)\,dx -\int_{v_Lt}^{v_Rt} \eta(u(t,x)|\bar{u}(x/t))\,dx\right\}. \label{eq:rar}
\end{eqnarray}
The dissipation due to the shocks of $u$ cancels out. Therefore, using twice \eqref{ineq:relative} with equality, (once with $b=u_L$, and once with $b=u_R$), and the inequality of Lemma \ref{lem:rar}, we find:
$$
D\geq q(u(t,tv+);u_R)-q(u(t,tv-);u_L)-v\left( \eta(u(t,tv+)|u_R) -\eta(u(t,tv-)|u_L)\right).
$$
Integrating in time gives that 
\begin{eqnarray*}
&&\qquad \int_{0}^{t} \left\{q(u(t,tv+);u_R)-q(u(t,tv-);u_L)-v\left( \eta(u(t,tv+)|u_R) -\eta(u(t,tv-)|u_L) \right)
 \right\}\,dt\\
 &&\leq \int_0^t D\,dt\\
 && \leq \int_{v_Lt}^{vt} \eta(u(t,x)|u_L)\,dx+\int_{vt}^{v_Rt} \eta(u(t,x)|u_R)\,dx -\int_{v_Lt}^{v_Rt} \eta(u(t,x)|\bar{u}(x/t))\,dx\\
 &&\leq \int_{v_Lt}^{vt} \left( \eta(\bar{u}(x/t))-\eta(u_L)+ \eta'(\bar{u}(x/t))(u-\bar{u}(x/t))-\eta'(u_L)(u-u_L)  \right)\,dx\\
 &&\qquad \qquad +\int_{vt}^{v_Rt} \left( \eta(\bar{u}(x/t))-\eta(u_R)+ \eta'(\bar{u}(x/t))(u-\bar{u}(x/t))-\eta'(u_R)(u-u_R)  \right)\,dx\\
 &&\leq C\delta |u_L-u_R| t.
\end{eqnarray*}
\end{proof}

For the sake of completeness, we  conclude  this section by proving that Assumptions \ref{assum} are  always verified for the isentropic Euler  system \ref{eq:Euler}.
\begin{lemma}
Consider the  system \eqref{eq:Euler}, with state set \eqref{InvReg}. Then  Assumptions \ref{assum} are verified.
\end{lemma}
\begin{proof}
The properties  \ref{assum} (a), (b), (c) are well known (see \cite{Serre_Book} for instance). Denote the flux function
$
f(\rho,\rho v)=(\rho v, \rho v^2+\rho^\gamma), 
$ 
and   the conservative quantities $(u_1,u_2)=(\rho, \rho v)$. Then 
$$
f_1(u_1,u_2)=u_2, \qquad f_2(u_1,u_2)=\frac{u_2^2}{u_1}+ u_1^\gamma.
$$
Note that $f_1$ is linear in $u$, and $f_2$ is convex. Hence, for any vector $\ell$ (left eigenvector of $f'(v)$ or not):
$$
\ell\cdot f=\ell_1 f_1+\ell_2 f_2
$$
is convex if $\ell_2\geq0$, and concave if $\ell_2\leq 0$. This proves \ref{assum} (d). 
\vskip0.3cm
The eigenvalues of $f'$ are given by the formula $\lambda_{\pm}=v\pm \sqrt{\gamma\rho^{\gamma-1}}$ (see  \cite{Serre_Book}). But from \eqref{InvReg},
$|v|\leq C$, and  $c_1 \rho^{(\gamma-1)/2}\leq 2C$. This shows  \ref{assum} (e). 


\end{proof}




\section{Modified front tracking algorithm}\label{section:front}

In the proof of Proposition \ref{main_prop}, the function $\psi$ will be defined through a modification of the  front tracking algorithm. For an excellent introduction to the front tracking algorithm, we refer the reader to Chapter 14 of Dafermos's  book \cite{dafermos_big_book} and also the succinct paper of Baiti-Jenssen \cite{MR1492096}. These two references, together, make an excellent introduction.

For completeness, we include here a brief description of the front tracking algorithm as we use it. In this paper, we do not make use of any of the convergence properties of the front tracking algorithm or related analysis. We use instead the fact that the algorithm gives a sequence of functions with uniformly bounded total variation. 

 For the construction of the $\psi$ we are about to give, the modification to the front tracking algorithm (as presented in Baiti-Jenssen \cite{MR1492096}) consists  in 
changing the velocity of the shocks.  The shocks move with an artificial velocity dictated by the shift functions of Proposition \ref{shift_existence_prop}, instead of moving with the Rankine-Hugoniot speed. 
Thus, performing analysis on our version of the front tracking algorithm is nearly identical to performing analysis on the front tracking algorithm when shocks move with Rankine-Hugoniot speed.

\vskip0.1cm
We now give the details of the construction of $\psi$, following Baiti-Jenssen \cite{MR1492096}. The main idea is to take a piecewise-constant approximation of the initial data, solve (approximately) all of the local Riemann problems within the class of piecewise-constant functions, until a time when two of the Riemann problems interact. Then, the procedure is repeated: the local Riemann problems are again solved, etc. 

The key point is to show that the number of wave-fronts (i.e., curves of discontinuity in time-space) remains finite, so this inductive process does not terminate in finite time. This is done by using two different Riemann solvers: an accurate Riemann solver is used to continue the solution in time after the interaction of two wavefronts when the product of the two strengths of the wavefronts is large. With the accurate solver, the number of wave fronts in the solution might increase. On other hand, when the product of the two strengths of the colliding wavefronts is small, a simplified Riemann solver is used which will prevent an explosion in the number of wavefronts. The key is that the accurate Riemann solver will only need to be used a small number of times, keeping the number of wavefronts in our solution finite. 

Recall that given a Riemann problem with two constant states $u_-$ and $u_+$ sufficiently close, a solution with at most three constant states, connected by either shocks or rarefaction fans, can always be found. More precisely, there exist $C^2$ curves $\sigma\mapsto T_i(\sigma)(u_-)$, $i=1,2$, parametrized  by arclength, such that
\begin{align}\label{RP_curves}
u_+=T_2(\sigma_2)\circ T_1(\sigma_1)(u_-),
\end{align}
for some $\sigma_1$ and $\sigma_2$. We define $u_0\coloneqq u_-$ and 
\begin{align}
u_1\coloneqq T_1(\sigma_1)(u_0),\\
u_2\coloneqq T_2(\sigma_2)\circ T_1(\sigma_1)(u_0).
\end{align}

We use the convention that, when $\sigma_i$ is positive (negative) the states $u_{i-1}$ and $u_i$ are separated by an i-shock (i-rarefaction) wave. Further,  the strength of the i-wave is defined as $\abs{\sigma_i}$. 

For given initial data $u^0$, let $\psi_\nu^0$ be a sequence of piecewise-constant functions approximating $u^0$ in $L^2$ on $(-R,R)$. (We will choose $\nu$ later such as to give us the required $\psi=\psi_\nu$.) Let $N_\nu$ be the number of discontinuities in the function $\psi_\nu$ and choose a parameter $\delta_\nu$ controlling the maximum strength of the (approximate) rarefaction fronts.

We now introduce the two Riemann solvers. One will be used when the product of the strengths of the colliding waves is large, the other will be used when the product of the strengths is small or one of the incoming waves is non-physical (also known as a pseudoshock).

\subsection{The Riemann solvers}
The Riemann solvers will use \emph{non-physical waves} (also known as \emph{pseudoshocks}). These are waves connecting two states (let's call them $u_-$ and $u_+$), and traveling with a fixed velocity $\hat{\lambda}>0$ defined in Proposition \ref{shift_existence_prop}. Therefore, it is greater than all characteristic speeds on $\mathcal{V}$ and greater than the speed of the shifts (which have a uniform bound on their speeds). We define this non-physical wave to have strength $\abs{\sigma}\coloneqq \abs{u_--u_+}$ and we say it belongs to the third wave family. Remark that since all non-physical waves travel with the same speed $\hat{\lambda}$, they cannot interact with each other.

Assume that at a positive time $\bar{t}$, there is an interaction at the point $\bar{x}$ between two waves of families $i_\alpha,i_\beta$ and strengths $\sigma_\alpha',\sigma_\beta'$, respectively, with $1\leq i_\alpha,i_\beta \leq 3$. Let $\sigma_\alpha'$ denote the left incoming wave. Let $u_-$, $u_+$ be the Riemann problem generated by the interaction, and let $\sigma_1,\sigma_2$ and $u_0,u_1,u_2$ be defined as in \eqref{RP_curves}. Finally, we can now define the accurate and simplified Riemann solvers.

(A) \emph{Accurate solver}: If $\sigma_i<0$ then we let 
\begin{align}\label{control_rarefaction_932020}
p_i\coloneqq \left \lceil{\sigma_i/\delta_\nu}\right \rceil ,
\end{align}
where $\left \lceil{s}\right \rceil$ denotes the smallest integer number greater than $s$. For $l=1,\ldots,p_i$ we define
\begin{align}
u_{i,l}\coloneqq T_i(l\sigma_i/p_i)(u_{i-1}),\hspace{.7in} x_{i,l}(t)\coloneqq \bar{x}+(t-\bar{t})\lambda_i(u_{i,l}).
\end{align}
On the other hand, if $\sigma_i>0$, we define $p_i\coloneqq 1$ and 
\begin{align}
u_{i,l}\coloneqq u_i ,\hspace{.7in} x_{i,l}(t)\coloneqq h_i(t).
\end{align}
Here, $h_i$ is the shift function coming from \Cref{shift_existence_prop}. Within the context of \Cref{shift_existence_prop}, we take $u_L=u_{i-1}$ and $u_R=u_i$. 

Then, we define the approximate solution to the Riemann problem as follows:

\begin{align}\label{accurate_RP_sol}
v_a(t,x)\coloneqq
\begin{cases}
u_-, &\mbox{ if } x < x_{1,1}(t),\\
u_+ ,&\mbox{ if } x>x_{2,p_2}(t),\\
u_i,&\mbox{ if } x_{i,p_i}(t)<x<x_{i+1,1}(t),\\
u_{i,l},&\mbox{ if } x_{i,l}(t)<x<x_{i,l+1}(t)\hspace{.3in}(l=1,\ldots,p_i-1).
\end{cases}
\end{align}
Note that thanks to the two last properties of Proposition \ref{shift_existence_prop}, we have : $x_{i,p_i}(t)<x_{i+1,1}(t)$ for all $t>0$,  so the function is well defined.

(B) \emph{Simplified solver}: for each $i=1,2$ let $\sigma_i''$ be the sum of the strengths of the strengths of all incoming i-waves. Define
\begin{align}
u'\coloneqq T_2(\sigma_2'')\circ T_1(\sigma_1'')(u_-).
\end{align}

Let $v_a(t,x)$ be the approximate solution of the Riemann problem $(u_-,u')$ given by \eqref{accurate_RP_sol}. Remark that in general $u'\neq u_+$ and thus we are introducing a non-physical front between these states. Hence, we define the simplified solution as follows:
\begin{align}\label{simple_RP_sol}
v_s(t,x)\coloneqq
\begin{cases}
v_a(t,x), &\mbox{ if } x-\bar{x} < \hat{\lambda}(t-\bar{t}),\\
u_+ ,&\mbox{ if } x-\bar{x} > \hat{\lambda}(t-\bar{t}).
\end{cases}
\end{align}

Notice that by construction, the simplifed solution to the Riemann problem contains at most two physical waves and an additional non-physical wave. Thus, by strategically employing the simplified solver for small collisions, we can prevent an explosion in the number of wavefronts.

\subsection{Construction of the approximate solutions}
Given $\nu$ we construct the approximate solution $\psi_\nu(t,x)$ as follows. At time $t=0$ all of the Riemann problems in $\psi_\nu^0$ are solved accurately as in (A) (the accurate solver). By slightly perturbing the speed of a wave if necessary, we can ensure that at each time we have at most one collision, which will involve only two wavefronts. Suppose that at some time $t>0$ there is a collision between two waves from the $i_\alpha$th and $i_\beta$th families. Denote the strengths of the two waves by $\sigma_\alpha$ and $\sigma_\beta$, respectively. The Riemann problem generated by this interaction is solved as follows. Let $\epsilon_\nu$ be a fixed small parameter which will be chosen later. 
\begin{itemize}
\item if $\abs{\sigma_\alpha\sigma_\beta}>\epsilon_\nu$ and the two waves are physical, then we use the accurate solver (A);
\item if $\abs{\sigma_\alpha\sigma_\beta}<\epsilon_\nu$ and the two waves are physical, or one wave is non-physical, then we use the simplified solver (B).
\end{itemize}

By the following Lemma, for any $\epsilon_\nu$ this algorithm will yield an approximate solution defined for all times $t>0$.

\begin{lemma}[from \protect{\cite[Lemma 2.1]{MR1492096}}]
The number of wavefronts in $\psi_\nu(t,x)$ is finite. Hence, the approximate solutions $\psi_\nu$ are defined for all $t>0$. 
\end{lemma}

This Lemma is stated and proved in \cite[Lemma 2.1]{MR1492096} for piecewise constant front tracking solutions where shocks move according to Rankine-Hugoniot. We do not repeat the proof here, because using shifts in the front tracking algorithm (instead of Rankine-Hugoniot speeds) does not impact the proof. The proof is identical.
\vskip0.3cm

We introduce the total variation of $\psi_\nu$ as 
\[
L(t)=\sum |\sigma_i|=\hbox{TV}(\psi_\nu)(t),
\]
namely the sum of the strengths of all jump discontinuities that cross the $t$-time line. Clearly, $L(t)$ stays constant along time intervals between consecutive collisions of fronts and changes only across points of wave interaction. 

A $j$-wave and an $i$-wave, with the former crossing the $t$-time line to the left of the latter, are called {\it approaching} when either $i<j$, or $i=j$ and at least one of these waves is a shock.
We recall then the definition of the  potential for wave interactions
\[
Q(t)=\sum_{i,j:\hbox{approaching waves}}|\sigma_i| |\sigma_j|,
\]
where the summation runs over all paires of approaching waves, with strengths $|\sigma_i|$ and $|\sigma_j|$, which cross the $t$-line. 
Let us summarize some well known fact of the front tracking method which are still valid in our situation. 
\begin{proposition}\label{prop:delta}
There exists $\kappa>0$ such that for any $\eps$ small enough, the following is true.

The functional $L(t)+\kappa Q(t)$ is decreasing in time. Moreover, for any time $t$ where waves with strength $|\sigma_i|$ and $|\sigma_j|$ interact the jump of $Q$ at this time verifies
$$
\Delta L(t)+\kappa \Delta Q(t)\leq -\frac{\kappa}{2}|\sigma_i| |\sigma_j|.
$$
Especially, there exists a constant $C>0$, such that  for every $\nu>0$, $T>0$:
\begin{eqnarray*}
&&\|\psi_\nu\|_{L^\infty(0,T, BV(\RR))}\leq 2\eps,\\
&&\norm{\psi_\nu(t,\cdot)-\psi_\nu(s,\cdot)}_{L^1}\leq C\abs{t-s},\qquad 0<s<t<T,\\
&&  \text{The function } \psi_\nu \text{ verifies the Condition \ref{H} with constant  } C.
\end{eqnarray*}
\end{proposition}
\begin{proof}
The definitions  of the functional $Q$ and $L$ do not depend on the propagation speed of the waves, as long as one  verifies the rules that only approaching waves can interact in the future, and (after interaction) interacting waves will not be approaching anymore in the future. These two rules are still valid, thanks to the separation of wave speeds by families in Proposition \ref{shift_existence_prop}. Therefore, the evolution rule of $L(t)$ and $Q(t)$ after each  collision depends only on the Riemann solvers which are identical to the real front tracking algorithm. So we recover the estimates involving $\Delta Q(t)$ and $\Delta L(t)$ in the exact same way as the original front tracking method (see \cite{Bbook}).
\vskip0.3cm
This estimate shows that $L(t)+\kappa Q(t)$ is a non increasing function in time. We have also that $Q(t)\leq \eps L(t)$. So for every time $t>0$
$$
L(t)\leq L(t)+\kappa Q(t)\leq L(0)+\kappa Q(0)\leq (1+\kappa \eps)\eps,
$$
which provides the uniform $BV$ bound for $\eps$ small enough.
\vskip0.3cm
As in \cite[(7.79)]{Bbook}, we have that   
\begin{eqnarray*}
&&\norm{\psi_\nu(t,\cdot)-\psi_\nu(s,\cdot)}_{L^1}\leq \mathcal{O}(1) |t-s| \left(\sup_{\tau\in \R^+} L(\tau) \right)[\text{maximum speed}]\\
&&\qquad\qquad \leq C|t-s|,
\end{eqnarray*}
since the maximum speed is uniformly bounded by $\hat{\lambda}$ defined in Proposition \ref{shift_existence_prop}.
\vskip0.3cm
The proof of the last statement is identical to  \cite[Lemma 7.3]{Bbook}, since it depends only on the interaction rules, and on the finite speed of propagation. 

\end{proof}

\vskip0.3cm
For every time $r>0$, we denote by $\mathcal{P}(r)$ the set of $i$ corresponding to non-physical waves. For the same reasons, the following lemma is unchanged from {\cite[Lemma 3.1]{MR1492096}}.


\begin{lemma}[from \protect{\cite[Lemma 3.1]{MR1492096}}]
\label{control_nonphys_932020}
If 
\begin{align}\label{want_942020}
\lim_{\nu\to\infty} \epsilon_\nu \Bigg(N_\nu+\frac{1}{\delta_\nu}\Bigg)^k=0,
\end{align}
for every positive integer $k$, then the total strength of non-physical waves in $\psi_\nu$ goes to zero uniformly in $t$ as $\nu\to\infty$:
$$
\sup_{r\in[0,T]}\sum_{i\in \mathcal{P}(r)} |\sigma_i|\to 0, \text{ when } \nu \to 0.
$$
\end{lemma}

\Cref{control_nonphys_932020} is based on \cite[Lemma 3.1]{MR1492096}, where the result is given for piecewise-constant front tracking solutions with shocks moving with Rankine-Hugoniot speed. A proof is also provided in \cite{MR1492096}. The proof of our \Cref{control_nonphys_932020}, where shocks move according to shift functions, is identical.

\section{The weight function $a$}\label{sec:Geng}

For any pairwise interaction between two small (shock or rarefaction) waves, one has the following estimates (see \cite{Bbook, gl, sm}). See Figure \ref{fig1}.

\begin{proposition}\label{prop1}
Call $\sigma', \sigma''$ the strengths of two interacting wave-fronts, and let 
$\sigma_1, \sigma_2$ be the strengths of the outgoing waves  of the first and second family. $\sigma$ takes positive sign on a shock and negative value on a rarefaction front. 

Then there exists a constant $C_0$ (uniformly valid for $u\in B_\eps(d)$) such that
\begin{itemize}
\item If both $\sigma'$ and $\sigma''$ belong to the first family, then
\bel{21}
|\sigma_1-(\sigma'+\sigma'')|+|\sigma_2|~\leq~C_0\,|\sigma'\sigma''|(|\sigma'|+|\sigma''|).\eeq
\item If $\sigma''$ is a 1-wave and $\sigma'$ is a 2-wave, then
\bel{299}
|\sigma_1-\sigma''|+|\sigma_2-\sigma'|~\leq~C_0\,|\sigma'\sigma''|(|\sigma'|+|\sigma''|).\eeq
\item If both $\sigma'$ and $\sigma''$ belong to the second family, then
\bel{22}
|\sigma_1|+|\sigma_2-(\sigma'+\sigma'')|~\leq~C_0\,|\sigma'\sigma''|(|\sigma'|+|\sigma''|).\eeq
\end{itemize}
\end{proposition}

Remark that $B_\eps(d)$ is in the context of \Cref{theo_Bressan}.

Here we can always choose $\varepsilon$ small enough, especially smaller than the $\eps$ of Proposition \ref{shift_existence_prop},  and such that 
\beq\label{c0}
C_0\varepsilon\leq 1.
\eeq

\begin{figure}[tb] \centering
		\includegraphics[width=.3\textwidth]{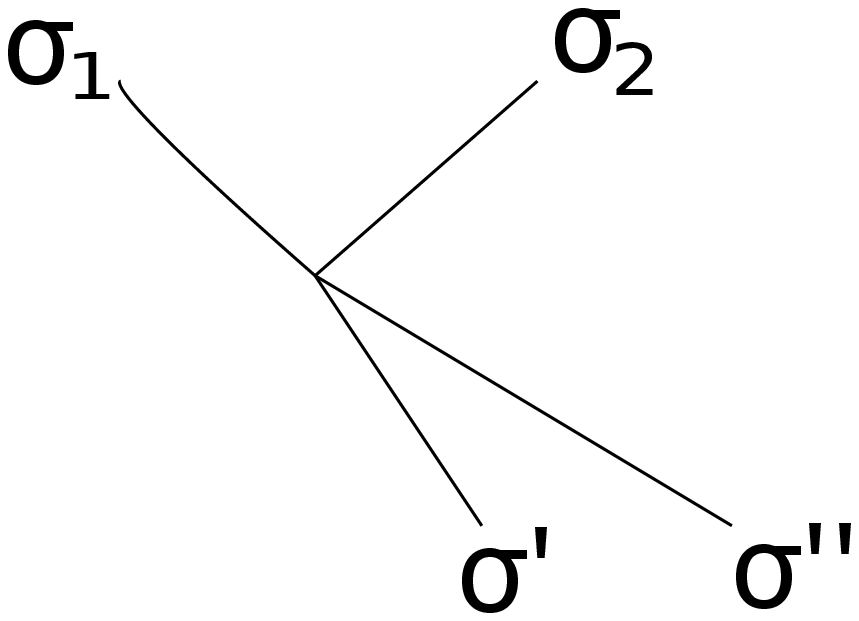}\hfill
		\includegraphics[width=.27\textwidth]{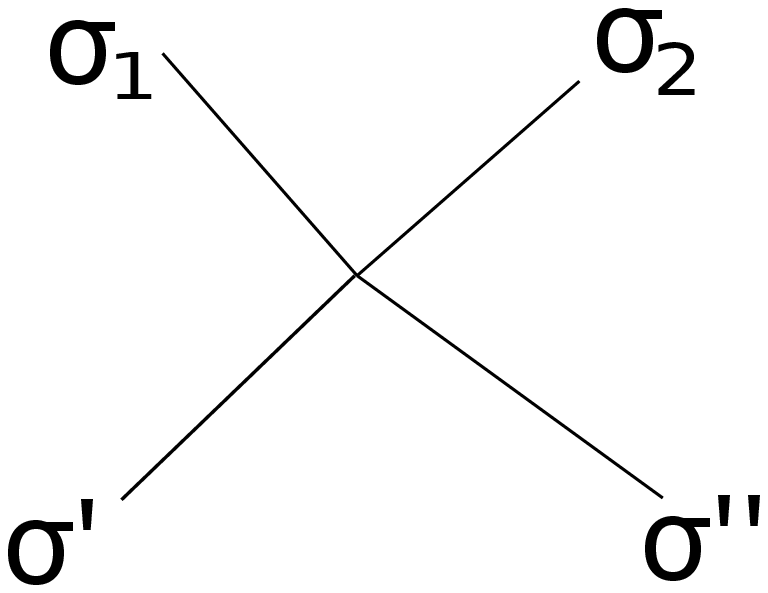}\hfill
		\includegraphics[width=.26\textwidth]{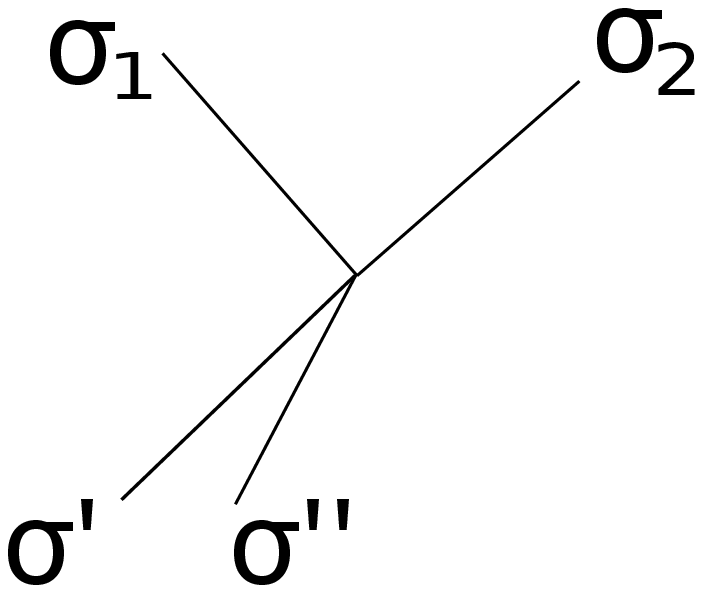}\hfill
		\caption{Pairwise interactions. Each line segment represent a shock or a rarefaction jump. To distinguish from each other, we represent 1-waves by line segment with negative speeds, and 2-waves
		with positive speeds. 
		The middle picture is for the head-on interaction. The other two pictures are for overtaking interactions.\label{fig1}}
	\end{figure}

\vskip0.3cm
We now define the following measure $\mu(t,\cdot)$ as a sum of Dirac measures in $x$:
\begin{eqnarray*}
&&\mu(t,x)= -\sum_{i: 1-\text{shock}} \sigma_i \delta_{\{x_i(t)\}}+\sum_{i: 2-\text{shock}} \sigma_i \delta_{\{x_i(t)\}}\\
&&\qquad = -\sum_{i: 1-\text{wave}} (\sigma_i)_+ \delta_{\{x_i(t)\}}+\sum_{i: 2-\text{wave}} (\sigma_i)_+ \delta_{\{x_i(t)\}}.
\end{eqnarray*}
The weight function is then defined as 
\begin{equation}\label{eq:a}
a(t,x)=1+ C\left(L(t)+\kappa Q(t) +\int_{-\infty}^x \mu(t,z)\,dz\right),
\end{equation}
where the constant $C$ is defined in Proposition \ref{shift_existence_prop}.
\vskip0.3cm
Note that the function $a$ is piecewise constant, with discontinuities only along shock curves. In particular it is constant across rarefaction curves and pseudoshock curves. 
We show that the function $a$ has the following properties.
\begin{proposition}\label{prop:a}
There exists $C_0>0$, such that for every $\eps>0$ small enough,
$$
|a(t,x)-1|\leq C_0\eps.
$$
For every time without wave interaction, and for every $x$ such that a 1-shock  $\sigma_i$ is located at  $x=x_i(t)$: 
\begin{equation}\label{sig1}
1-2C|\sigma_i|\leq\frac{a(t,x_i(t)+)}{a(t,x_i(t)-)}\leq 1-\frac{C}{2}|\sigma_i|,
\end{equation}
\vskip0.1cm
For every time without wave interaction, and for every $x$ such that a 2-shock  $\sigma_i$ is located at  $x=x_i(t)$:
\begin{equation}\label{sig2}
1+\frac{1}{2}C|\sigma_i|\leq\frac{a(t,x_i(t)+)}{a(t,x_i(t)-)}\leq 1+{2}C|\sigma_i|,
\end{equation}
\vskip0.1cm
For every time $t$ with a wave interaction, and almost every $x$:
\begin{equation}\label{eqa2}
a(t+,x)\leq a(t-,x).
\end{equation}
\end{proposition}
\begin{proof}
Note that $\|\mu(t)\|_{\mathcal{M}}\leq L(t)$, and
\[
L(0)+\kappa Q(0)=O(\varepsilon).
\]
Since $L(t)+\kappa Q(t)$ is decreasing,
it means that 
$$
|a(t,x)-1|=O(\varepsilon).
$$
So for $\eps$ small enough, $1/2<1/a(t,x(t)-)<2$.
Now:
$$
\frac{a(t,x_i(t)+)}{a(t,x(t)-)}-1=\frac{1}{a(t,x(t)-)}(a(t,x_i(t)+)-a(t,x(t)-))=\frac{1}{a(t,x(t)-)}C|\sigma_i| \alpha,
$$
with $\alpha=1$ if the shock $\sigma_i$ is a 2-shock, and $\alpha=-1$ if it is a 1-shock. This shows \eqref{sig1} and \eqref{sig2}.
\vskip0.3cm
Consider a time $t$ with a wave interaction. From the definition of the $a$ function,
\begin{equation}\label{eqa3}
\sup_\RR (a(t+,x)-a(t-,x))\leq C\left(\int_\RR |\mu(t+)-\mu(t-)|\,dx+(\Delta L(t)+\kappa \Delta Q (t))\right).
\end{equation}
Assume that the waves interacts at $x=x_0$. The interacting wave fronts are $\sigma'$ $\sigma''$ leading to outgoing waves $\sigma_1$, $\sigma_2$.
We study $\mu(t+)-\mu(t-)$  by considering separately all the possible kind of interactions.
\vskip0.3cm
 If the interaction involves a pseudoshock, then the strengths of the shocks are not affected, so
$\mu(t+)-\mu(t-)=0$, and since $\Delta L+\kappa \Delta Q\leq 0$, we have also  $a(t+,x)-a(t-,x)\leq 0$ for every $x\in \RR$. 
\vskip0.3cm
If the simplified solver is used, then  
$$
\mu(t+)-\mu(t-)=\delta_{\{x_0\}} [(\sigma'+\sigma'')_+-(\sigma')_+-(\sigma'')_+]\leq 0.
$$
We still have $a(t+,x)-a(t-,x)\leq 0$ in this case.
\vskip0.3cm

It remains to consider the cases involving the accurate solver. They correspond to the three cases of  Proposition \ref{prop1}.
\vskip0.1cm
\paragraph{\bf (i)} If $\sigma''$ is a 1-wave and $\sigma'$ is a 2-wave.
Using the definition of $\mu$ to justify the first equality below, the fact that $y\to(y)_+$ is Lipschitz with constant 1 for the  second inequality,  \eqref{299} for the third inequality, and Proposition \ref{prop:delta} for $\eps$ small enough to get the last inequality, we have:
\begin{eqnarray*}
&&|\mu(t+)-\mu(t-)|=\delta_{\{x_0\}}|(\sigma_2)_+-(\sigma_1)_+-((\sigma')_+-(\sigma'')_+)|\\
&&\qquad \leq \delta_{\{x_0\}}|(\sigma_2)_+-(\sigma')_+|+|(\sigma_1)_+-(\sigma'')_+|\\
&&\qquad \leq \delta_{\{x_0\}}(|\sigma_2-\sigma'|+|\sigma_1-\sigma''|)\\
&&\qquad \leq \delta_{\{x_0\}} C_0\,|\sigma'\sigma''|(|\sigma'|+|\sigma''|)\\
&&\qquad \leq  -(\Delta L(t)+\kappa \Delta Q (t))\delta_{\{x_0\}}.
\end{eqnarray*}
Remark also that we used \eqref{c0}. But using \eqref{eqa3} and the fact that $\Delta L(t)+\kappa \Delta Q (t)<0$, gives 
$$
\sup_\RR (a(t+,x)-a(t-,x))\leq 0.
$$

\vskip0.3cm
\paragraph{\bf (ii)} The cases \eqref{21} and \eqref{22} are similar. Let us do \eqref{21} in detail. Both $\sigma'$ and $\sigma''$ belong to the first family. 
In this case,
\begin{eqnarray*}
&&|\mu(t+)-\mu(t-)|=\delta_{\{x_0\}}|(\sigma_2)_+-(\sigma_1)_+-(-(\sigma')_+-(\sigma'')_+)]|\\
&&\qquad \leq \delta_{\{x_0\}}(|(\sigma_2)_+|+|(\sigma_1)_+-(\sigma')_+-(\sigma'')_+|)\\
&&\qquad \leq \delta_{\{x_0\}}(|(\sigma_2)|+|(\sigma_1)_+-(\sigma')_+-(\sigma'')_+|)\
\end{eqnarray*}
We need to separate cases depending on the nature of the incoming waves, and the kind of Riemann solver used.

\paragraph{\bf (ii)-1} If $\sigma'$ and $\sigma''$ are rarefactions, then $\sigma_1$ is also a rarefaction, and 
$$
|(\sigma_1)_+-(\sigma')_+-(\sigma'')_+|=0.
$$

\paragraph{\bf (ii)-2} If $\sigma'$ and $\sigma''$ are shocks, then $\sigma_1$ is also a shock, and 
$$
|(\sigma_1)_+-(\sigma')_+-(\sigma'')_+|=|\sigma_1-\sigma'-\sigma''|.
$$

\paragraph{\bf (ii)-3} Finally, if one of  $\sigma'$ and $\sigma''$ is a shock, let say $\sigma'$,  and the other a rarefaction, let say $\sigma''$. 
Then 
\begin{eqnarray*}
&&|(\sigma_1)_+-(\sigma')_+-(\sigma'')_+|=|(\sigma_1)_+-\sigma'|\leq |\sigma''-(\sigma_1)_-|+|\sigma_1-\sigma'-\sigma''|\\
&&\qquad \leq  -(\Delta L(t)+\kappa \Delta Q (t)),
\end{eqnarray*}
since $ |\sigma''-(\sigma_1)_-|\leq -\Delta L(t)+C_0\eps |\sigma'||\sigma''|$.

Gathering all the cases, we obtain:
$$
\sup_\RR (a(t+,x)-a(t-,x))\leq 0.
$$

\end{proof}


\section{Proof of Proposition \ref{main_prop}}\label{sec:Sam}


This section is devoted to the proof of Proposition \ref{main_prop}.
In our front-tracking procedure, we are stopping and restarting the clock every time there is a collision between waves (when the waves initiated from distinct Riemann problems). Weak solutions $u$ to \eqref{cl} naturally lies in $C^0(\RR^+; W^{-1,\infty}(\RR))$. Note that the  formulation of the entropy inequality \eqref{entropy_ineq_integral_form} holds with a boundary term for $t=0$, and this classically  implies that  $u$  is continuous in time at $t=0$ with values in $L^1_{\mathrm{loc}}(\RR)$. Because $L^1_{\mathrm{loc}}(\RR)$ is a strong topology, it implies that $\eta (u)$ is also continuous at $t=0$ in the same topology in $x$. 
However, because $\eta(u)$ verifies only Inequality \eqref{ineq:entropy},  $\eta(u)$ does not share this regularity in time for $t>0$. 
Therefore $\eta(u)$ is well defined only almost everywhere in time. 
 However, this technicality of stopping and restarting the clock at any time $t$ is not a real issue, and its resolution can be formalized with  the use of approximate limits as follows. For a reference on approximate limits, see \cite[p.~55-57]{MR3409135}.
 \begin{lemma}[A technical lemma on stopping and restarting the clock (from \protect{\cite[Lemma 2.5]{2017arXiv170905610K}})]
\label{ap_lim_lemma}
~\\
Let $u\in L^\infty (\mathbb{R}^+\times\mathbb{R})$ be a weak solution to \eqref{cl} with initial data $u^0$. Further, assume that $u$ is entropic for the entropy $\eta$, i.e. verifies \eqref{ineq:entropy} in the sense of distribution. Assume also that $u$ verifies the strong trace property (\Cref{defi_trace}). Then for all $v\in \Nuo$, and for all $c,d\in\mathbb{R}$ with $c<d$, the following approximate right- and left-hand limits 
\begin{align}
{\rm ap}\,\lim_{t\to {t_0}^{\pm}} \int\limits_c^d \eta(u(t,x)|v)\,dx
\end{align}
exist for all $t_0\in(0,\infty)$ and verify 
\begin{align}\label{ap_lim_ordering}
{\rm ap}\,\lim_{t\to {t_0}^{-}} \int\limits_c^d \eta(u(t,x)|v)\,dx\geq {\rm ap}\,\lim_{t\to {t_0}^{+}} \int\limits_c^d \eta(u(t,x)|v)\,dx.
\end{align}
Furthermore, the approximate right-hand limit exists at $t_0=0$ and verifies
\begin{align}\label{ap_lim_ordering0}
\int\limits_c^d \eta(u^0(x)|v)\,dx\geq {\rm ap}\,\lim_{t\to {t_0}^{+}} \int\limits_c^d \eta(u(t,x)|v)\,dx.
\end{align}
\end{lemma}

The proof of \Cref{ap_lim_lemma} follows exactly the proof of \cite[Lemma 2.5]{2017arXiv170905610K}. For this reason, we do not include a proof here.
\vskip0.3cm
We gather in the following lemma useful simple properties of the relative quantities.
\begin{lemma}\label{lem:relative}
For any  $\mathcal{O}$ open subset of $\Nu$ with $\overline{\mathcal{O}}\subset \Nu$,  there exists a constant $C>0$ such that 
\begin{eqnarray*}
&&|q(a;b)|\leq C\eta(a|b),\qquad \forall (a,b)\in \Nuo\times \overline{\mathcal{O}},\\
&&|q(a;b_1)-q(a;b_2)|\leq C|b_1-b_2|, \qquad \forall (b_1,b_2)\in \overline{\mathcal{O}}^2, \  a\in \Nu,\\
&&|\eta(a|b_1)-\eta(a|b_2)|\leq C|b_1-b_2|, \qquad \forall (b_1,b_2)\in \overline{\mathcal{O}}^2, \  a\in \Nu.
\end{eqnarray*}
\end{lemma}
\begin{proof}
Consider an open set $\mathcal{O}'$ such that $\overline{\mathcal{O}}\subset \mathcal{O}'$ and $\overline{\mathcal{O}'}\subset \Nu$.
Since both $f,q\in C^0(\Nuo)$ and $f'\in C^0(\overline{\mathcal{O}})$, $q(\cdot;\cdot)$ is uniformly bounded on $\Nuo\times \overline{\mathcal{O}}$. Moreover, 
from  Lemma \ref{l2_rel_entropy_lemma}, $\eta(\cdot,\cdot)$ is bounded above and below uniformly on $(\Nuo\setminus \mathcal{O}')\times \overline{\mathcal{O}}$.
Therefore there exists a constant such that the first equality holds for those values.
But from  the definition in \eqref{def_q}, $q(b;b)=\partial_1q(b,b)=0$ for all $b\in \Nu$. So using Lemma \ref{l2_rel_entropy_lemma}, and the fact that $q''\in C^0(\overline{\mathcal{O}'})$ we have that there exists a constant $C$ such that 
$$
|q(a;b)|\leq C|a-b|^2\leq \frac{C}{c^*} \eta(a|b),\qquad for (a,b) \in \overline{\mathcal{O}'}\times \overline{\mathcal{O}}.
$$
This proves the first inequality of the lemma.
\vskip0.3cm
From the definition of $q(\cdot;\cdot)$ in  \eqref{def_q}, denoting $h=q-\eta'f\in C^1(\overline{\mathcal{O}})$, we have for every $(a,b_1,b_2)\in \Nuo\times \overline{\mathcal{O}}^2$:
$$
|q(a;b_1)-q(a;b_2)|=|h(b_2)-h(b_1)+[\eta'(b_2)-\eta'(b_1)]f(a)|\leq \sup_{\overline{\mathcal{O}}}(|h'|+|\eta''|) (1+\sup_{\Nuo}|f|)|b_2-b_1|.
$$
The proof of the last statement is similar.
\end{proof}

\vskip0.3cm
We now prove Proposition  \ref{main_prop}.
First we fix  the  value $v$ to be bigger than both $\hat{\lambda}$ and the constant $C$ of Lemma \ref{lem:relative}.
Take $0<\eps<1/2$ small enough such that Theorem \ref{theo_Bressan}, Proposition \ref{shift_existence_prop}, and Proposition \ref{prop:delta}  hold true.  
For any initial value $u^0$, and wild solution $u\in \Sweak$, we consider the family of solutions $\psi_\nu$ of the modified front tracking method. We want now to choose a particular one. Fix  $T,R>0$, and $p\in \NN$. First we insure that the initial value verifies 
$$
\|u^0-\psi_\nu(0,\cdot)\|_{L^2(-R,R)}\leq \frac{1}{p}.
$$
This fixes $N_\nu$.  Then we fix $\delta_\nu=1/(pT)$. Thanks to Lemma \ref{control_nonphys_932020}, we can choose $\eps_\nu$ such that 
$$
\sup_{r\in[0,T]}\sum_{i\in \mathcal{P}(r)} |\sigma_i|\leq \frac{1}{pT}.
$$
We denote by $\psi$ the associated solution to the modified front tracking method $\psi_\nu$. Especially, it verifies
\begin{eqnarray}\label{eqini}
&& \|u(0,\cdot)-\psi(0,\cdot)\|_{L^2(-R,R)}\leq \|u^0-u(0,\cdot)\|_{L^2(-R,R)}+\frac{1}{p},\\ \label{eqinit2}
&& T \delta_\nu \sup_{t\in [0,T]} L(t)\leq \frac{1}{p},\\
&&T\sup_{r\in[0,T]}\sum_{i\in \mathcal{P}(r)} |\sigma_i|\leq \frac{1}{p}.\label{eqinit3}
\end{eqnarray}
Proposition  \ref{prop:delta} provides three of the four properties of Proposition  \ref{main_prop}. It remains only to show the control in $L^2$ of $\psi(t,\cdot)-u(t,\cdot)$. Recall that as in \Cref{section:front}, for every time $r>0$, we denote by $\mathcal{P}(r)$ the set of $i$ corresponding to non-physical waves.
\vskip0.3cm

Consider two successive  interaction times $t_j<t_{j+1}$ of the front tracking solution $\psi$. 
Let the curves of discontinuity  between the two times $t_j<t_{j+1}$ be $h_1,\ldots,h_N$ for $N\in\mathbb{N}$ such that 
\begin{align}
h_1(t)< \cdots < h_N(t),
\end{align}
for all $t\in(t_j,t_{j+1})$. 
We only work on the cone of information, so we define for all times $t$
\begin{align}
h_0(t)=-R+vt, \\
h_{N+1}=R-vt.
\end{align}
Note that 
 there are no interactions between wave fronts in $\psi$ and the cone of information (coming from $h_0$ and $h_{N+1}$). For any $t\in [t_j,t_{j+1}]$,
note that on $ Q=\{(r,x): t_j<r<t, h_i(r)<x<h_{i+1}(r)\}$, the function $\psi(r,x)= b$ is constant. Moreover, by construction,  the weight function $a(r,x)$ is also constant on this set. Therefore,  integrating \eqref{ineq:relative} on $Q$, and using the strong trace property of Definition \ref{defi_trace}, we find
\begin{eqnarray*}
&&\qquad\qquad\qquad {\rm ap}\,\lim_{s\to {t}^{-}} \int\limits_{h_i(t)}^{h_{i+1}(t)} a(t-,x)\eta(u(s,x)|\psi(t,x))\,dx\\
&&\leq {\rm ap}\,\lim_{s\to {t_j}^{+}} \int\limits_{h_i(t_j)}^{h_{i+1}(t_j)} a(t_j+,x)\eta(u(s,x)|\psi(x,t_j))\,dx
+\int_{t_j}^{t} (F^+_{i} (r)-F^-_{i+1}(r))\,dr,
\end{eqnarray*}
where 
\begin{eqnarray*}
&&F^+_i (r)=a(r,h_i(r)+)\left[q(u(r,h_i(r)+);\psi(r,h_i(r)+))-\dot{h}_{i}(r)\eta(u(r,h_i(r)+)|\psi(r,h_i(r)+))\right],\\
&&F^-_i(r)=a(r,h_i(r)-)\left[q(u(r,h_i(r)-);\psi(r,h_i(r)-))-\dot{h}_{i}(r)\eta(u(r,h_i(r)-)|\psi(r,h_i(r)-))\right].
\end{eqnarray*}

We sum in $i$, and  combine the terms corresponding to $i$ into one sum, and the terms corresponding to ${i+1}$ into another sum, to find
\begin{eqnarray*}
&&\qquad\qquad\qquad {\rm ap}\,\lim_{s\to {t}^{-}} \int\limits_{-R+vt}^{R-vt} a(t-,x)\eta(u(s,x)|\psi(t,x))\,dx\\
&&\leq {\rm ap}\,\lim_{s\to {t_j}^{+}} \int\limits_{-R+vt_j}^{R-vt_j} a(t_j+,x)\eta(u(s,x)|\psi(t_j,x))\,dx
+\sum_{i=1}^{N}\int_{t_j}^{t} (F^+_{i} (r)-F^-_{i}(r))\,dr,
\end{eqnarray*}
where we have used that $F^+_0\leq 0$ and $F^-_{N+1}\geq0$ thanks to the first statement of Lemma \ref{lem:relative}, the definition of $v$, and the fact that 
$\dot{h}_0=-v=-\dot{h}_{N+1}$. 
\vskip0.3cm
We decompose the sum into three sums, one corresponding to the shock fronts, one for the rarefaction fronts, and one for the pseudoshocks. 
Thanks to Proposition \ref{shift_existence_prop} and Proposition \ref{prop:a}, for any $i$ corresponding to a shock front:
$$
F^+_{i} (r)-F^-_{i}(r)\leq0, \qquad \text{for almost every } t_j<r<t.
$$
Denote $\mathcal{R}$ the set of $i$ corresponding to approximated rarefaction fronts. 
Then for any $i\in \mathcal{R}$ by construction, $a(h_i(r)+,r)=a(h_i(r)-,r)$. And from 
Proposition \ref{dissipation_discrete_rarefaction}, and \eqref{eqinit2}:
\begin{eqnarray*}
\sum_{i\in\mathcal{R}}\int_{t_j}^t  (F^+_{i} (r)-F^-_{i}(r))\,dr\leq C\delta_\nu(t-t_j)\sum_{i\in\mathcal{R}}|\sigma_i|
\leq C\delta_\nu (t-t_j)L(t)  \leq \frac{C}{pT} (t-t_j).
\end{eqnarray*}
Consider now the case when $i\in \mathcal{P}(r)$. Recall that pseudoshocks travel with supersonic (greater-than-characteristic) speed $\hat{\lambda}$. Thus, we must have that for almost every time $r$: $u(r,h_i(r)+)=u(r,h_i(r)-)$. This is because if $u(r,h_i(r)+)\neq u(r,h_i(r)-)$, then the shock $\big(u(r,h_i(r)+),u(r,h_i(r)-),\hat{\lambda}\big)$ would be traveling with speed greater than any of the eigenvalues of $Df$, a contradiction. By construction of the $a$ function, we know that $a$ does not have a jump across pseudoshocks,  so we have also $a(r,h_i(r)+)=a(r,h_i(r)-)$. Therefore, thanks to the second and third estimates of Lemma \ref{lem:relative}, 
$$
F^+_{i} (r)-F^-_{i}(r)\leq C |\psi(r,h_i(r)+)-\psi(r,h_i(r)-)|=C|\sigma_i|.
$$




 Then, from \eqref{eqinit3} we receive
$$
\sum_{i\in\mathcal{P}(r)} \int_{t_j}^t (F^+_{i} (r)-F^-_{i}(r))\,dr \leq C(t-t_j)\sum_{i\in\mathcal{P}(t)}|\sigma_i|\leq \frac{C(t-t_j)}{pT}.
$$
Gathering all the families of waves, we find:
\begin{eqnarray*}
&&\qquad\qquad\qquad {\rm ap}\,\lim_{s\to {t}^{-}} \int\limits_{-R+vt}^{R-vt} a(t-,x)\eta(u(s,x)|\psi(t,x))\,dx\\
&&\leq {\rm ap}\,\lim_{s\to {t_j}^{+}} \int\limits_{-R+vt_j}^{R-vt_j} a(t_j+,x)\eta(u(s,x)|\psi(t_j,x))\,dx
+\frac{C(t-t_j)}{pT}. 
\end{eqnarray*}
\vskip0.3cm
Consider now any $0<t<T$, and denote $0<t_1<\cdot\cdot\cdot<t_J$ the times of wave interactions before $t$, $t_0=0$, and $t_{J+1}=t$. Using the convexity of $\eta$, 
 Lemma \ref{ap_lim_lemma}, and \eqref{eqa2}  we find:
\begin{eqnarray*}
&& \qquad \int\limits_{-R+vt}^{R-vt} a(t,x)\eta(u(t,x)|\psi(t,x))\,dx- \int\limits_{-R}^{R} a(0,x)\eta(u(0,x)|\psi(0,x))\,dx\\
&&\leq  {\rm ap}\,\lim_{s\to {t}^{+}}\int\limits_{-R+vt}^{R-vt} a(t,x)\eta(u(s,x)|\psi(t,x))\,dx- \int\limits_{-R}^{R} a(0,x)\eta(u(0,x)|\psi(0,x))\,dx\\
&&\leq \sum_{j=1}^{J+1} \left(  {\rm ap}\,\lim_{s\to {t_j}^{+}}\int\limits_{-R+vt_j}^{R-vt_j} a(t_j-,x)\eta(u(s,x)|\psi(t,x))\,dx\right.\\
&&\qquad\qquad\qquad \qquad\qquad\left.- {\rm ap}\,\lim_{s\to {t_{j-1}^+}}\int\limits_{-R+vt_{j-1}}^{R-vt_{j-1}} a(t_{j-1}-,x)\eta(u(s,x)|\psi(t,x))\,dx\right)\\
&&\leq \sum_{j=1}^{J+1} \left(  {\rm ap}\,\lim_{s\to {t_j}^{-}}\int\limits_{-R+vt_j}^{R-vt_j} a(t_j-,x)\eta(u(s,x)|\psi(t,x))\,dx\right.\\
&&\qquad\qquad\qquad \qquad\qquad\left.- {\rm ap}\,\lim_{s\to {t_{j-1}}^{+}}\int\limits_{-R+vt_{j-1}}^{R-vt_{j-1}} a(t_{j-1}-,x)\eta(u(s,x)|\psi(t,x))\,dx\right)\\
&&\leq \sum_{j=1}^{J+1} \left(  {\rm ap}\,\lim_{s\to {t_j}^{-}}\int\limits_{-R+vt_j}^{R-vt_j} a(t_j-,x)\eta(u(s,x)|\psi(t,x))\,dx\right.\\
&&\qquad\qquad\qquad \qquad\qquad\left.- {\rm ap}\,\lim_{s\to {t_{j-1}}^{+}}\int\limits_{-R+vt_{j-1}}^{R-vt_{j-1}} a(t_{j-1}+,x)\eta(u(s,x)|\psi(t,x))\,dx\right)\\
&&\leq  \sum_{j=1}^{J+1} C\frac{(t_j-t_{j-1})}{Tp}\leq C\frac{t}{Tp}\leq \frac{C}{p}.
\end{eqnarray*}
Using that $|a-1|<1/2$  and  \eqref{eqini}, we get that for every $0<t<T$
$$
 \int\limits_{-R+vt}^{R-vt}\eta(u(t,x)|\psi(t,x))\,dx\leq  2\|u^0-u(0,\cdot)\|^2_{L^2(-R,R)}+\frac{C}{p}.
$$
Choosing $p$ big enough such that $C/p<1/m$ gives the result.

%
%

\appendix
\section{Proof of Lemma \ref{lem:limit}}

For any $R>0$, Consider $B_{-1}= L^1(-R,R)$, $B_0=L^2(-R,R)$, $B_1=BV(-R,R)$. These Banach Spaces are nested into each other,
and $B_1$ is injected in $B_0$ compactly. Moreover, the $\psi_n$ are uniformly bounded in $L^\infty(0,R;B_1)$ and, from \eqref{trucmuche}, $\{\partial_t\psi_n\}_{n\in\NN}$ is uniformly bounded in $L^\infty(0,R;B_{-1})$. Therefore, thanks to the Aubin Lions lemma, 
$\{\psi_n\}_{n\in\NN}$ is precompact in $C^0(0,R;L^2(-R,R))$. Therefore there exists a subsequence, still denoted $\{\psi_n\}$, and $\psi\in L^\infty(\RR^+\times \RR)$ such that  $\psi_n$ converges to $\psi$ in $C^0(0,R;L^2(-R,R))$ for all $R>0$. Obviously, the convergence holds also in $L^2((0,R)\times(-R,R))$.  Then, by a diagonal argument, we can re-extract a subsequence such that the convergence holds for almost every $(t,x)\in \RR^+\times\RR$. The function $\psi$ still verifies $\|\psi\|_{L^\infty(\R^+;BV(\RR))}\leq C$.  Thus, we can choose a representative of $\psi$ that is right-continuous (in $x$).

\vskip0.3cm
It remains to show that $\psi$ verifies the Bounded Variation Condition of Definition \ref{tame_def}  for $\delta\coloneqq \frac{1}{\hat{\lambda}}$.
Let $\gamma$, and $a',b'$ be a space-like curve as in Definition \ref{def:space}. We will now show that the function $x\mapsto \psi(\gamma(x),x)$ has bounded variation.  Let any points $P_i=(\gamma(x_i),x_i)$ be given, with $x_0<x_1<\cdots<x_m$. At each time $t_i=\gamma(x_i)$, we know that $\psi_n(t_i,\cdot)$ will converge to $\psi(t_i,\cdot)$ almost everywhere. Due to $\psi(t_i,\cdot)$ being right-continuous for each $i$, for any $\epsilon>0$ we can find $x_i'\in[x_i,x_i+\epsilon]$ such that
\begin{align}
\abs{\psi(t_i,x_i')-\psi(t_i,x_i)}< \epsilon,& \hspace{.4in} \lim_{n\to\infty} \psi_n(t_i,x_i')=\psi(t_i,x_i'),\label{first_line_bv_cond_proof}\\
\hat{\lambda}\abs{t_i-t_{i-1}}<x_i'-x_{i-1}',& \hspace{.4in} \hat{\lambda}t_m<b-x_m,  \hspace{.4in}  \hat{\lambda}t_0<x_0-a,\label{second_line_bv_cond_proof}
\end{align}
for constants $a$ and $b$ to be chosen momentarily.

Define $\gamma'$ to be the polygonal line with vertices at the points $P_i'\coloneqq (t_i,x_i')$ and also define $\bar{\gamma}$ to be the constant map $\bar{\gamma}(x)=0$ for all $x\in(a,b)$, where $a$ and $b$ are from \eqref{second_line_bv_cond_proof}.

Then, by choosing $a$ and $b$ such that $a'-a>0$ is suitably large and $b-b'>0$ is suitably large, the properties \eqref{first_line_bv_cond_proof} and \eqref{second_line_bv_cond_proof} imply that $\bar{\gamma}$ dominates $\gamma'$ (as in \Cref{domination_def}). 

Since $\psi_n$ verifies uniformly the property of Condition \ref{H},  from \eqref{truc}, we have
\begin{align}
\sum_{i=1}^m\abs{\psi(t_i,x_i)-\psi(t_{i-1},x_{i-1})}&\leq 2m\epsilon +\sum_{i=1}^m\abs{\psi(t_i,x_i')-\psi(t_{i-1},x_{i-1}')}\\
&\leq 2m\epsilon +\limsup_{n\to\infty}(\mbox{Tot. Var.}\{\psi_n;\gamma'\})\\
&\leq 2m\epsilon +C\limsup_{n\to\infty}(\mbox{Tot. Var.}\{\psi_n(0,\cdot);(a,b)\}).
\end{align}
Recall then that the total variation of $\psi_n(0,\cdot)$ on the interval $(a,b)$ is uniformly bounded in $n$. Then, since $\epsilon>0$ and the points $P_i$ were arbitrary, we conclude that $\psi$ verifies the Bounded Variation Condition (\Cref{tame_def}).

\bibliographystyle{apalike}
\bibliography{references09-31}
\end{document}